\documentclass{amsart}
\usepackage{amsmath,amsthm,amsfonts,amssymb,mathtools,mathrsfs,url}
  \usepackage{paralist}
  \usepackage{graphics}
    \usepackage{epsfig}
\usepackage{graphicx}
\usepackage{hyperref}
\usepackage{graphicx}
\usepackage{amsmath}
\usepackage{amssymb}
\usepackage[utf8]{inputenc}
\usepackage{subcaption}
\usepackage{enumitem}
\usepackage{url}
\usepackage{xcolor}
\usepackage[toc,page]{appendix}
\usepackage{stmaryrd}
\usepackage{mathrsfs}
\usepackage{mathtools}
\usepackage[abbrev]{amsrefs}
\usepackage{tikz}
\usetikzlibrary{patterns.meta}
\usepackage{pgfplots}
\usepackage{lmodern}
\usepackage{cancel}

\pgfplotsset{every axis/.append style={
                    axis x line=middle,    
                    axis y line=middle,    
                    axis line style={->}, 
                    xlabel={$x$},          
                    ylabel={$v$},          
            }}

\numberwithin{equation}{section}

\hypersetup{urlcolor=blue, citecolor=green}

\mathtoolsset{showonlyrefs=true}

\calclayout

\newcommand{\R}{\mathbb R} 
\newcommand{\N}{\mathbb N}

\newcommand{\E}{\mathbb E}

\newcommand{\1}{\mathbf 1}
\newcommand{\D}{\mathrm{d}}

\newcommand{\hop}{\vskip.3cm\noindent} 
\newcommand{\hip}{\vskip.1cm\noindent}

\newtheorem{thm}{Theorem}[section]   
\newtheorem{cor}[thm]{Corollary}
\newtheorem{prop}[thm]{Proposition}
\newtheorem{lem}[thm]{Lemma}
\newtheorem{rema}[thm]{Remark}

\DeclareRobustCommand{\SkipTocEntry}[5]{}

\title[]{A spectral approach to the narrow escape problem in
  two-dimensional domains}
\author{Louis Carillo\textsuperscript{1,2,3}, Tony Lelièvre\textsuperscript{1,2}, Thomas Normand\textsuperscript{2,1}, and Urbain Vaes\textsuperscript{2,1}}
\thanks{\textsuperscript{1}CERMICS, CNRS, ENPC, Institut Polytechnique de Paris, Marne-la-Vallée, France. \textsuperscript{2}MATHERIALS project-team, Inria Paris, France. \textsuperscript{3}E-mail address: \href{mailto:louis.carillo@enpc.fr}{louis.carillo@enpc.fr}}
\date{}

\begin{document}

\begin{abstract}
  We study the law of the exit time and exit point of a Brownian
  motion in a two-dimensional domain with reflecting boundary
  conditions, except on small disjoint exit windows through which the
  stochastic process can escape the domain. In the limit of infinitely small exit
  windows, it is natural to assume that the process starts from the
  quasi-stationary distribution. In this setting, we obtain a precise description of
  the exit event.
\end{abstract}

\maketitle

\tableofcontents



\section{Introduction}

The {\em narrow escape problem} consists in characterizing the exit
event of a stochastic process from a domain in the asymptotic regime
where the exit regions are infinitely small. This problem has been
extensively considered in the physics literature, see for
example~\cite{bib1,bib2,grigoriev2002kinetics,schuss2007narrow} for a
few representative works. This problem naturally occurs in biology to
model the exit of molecular species from a cell through narrow channels
inserted in the cell membrane. It can also be seen as a prototypical
example of a so-called entropic metastability (as opposed to energetic
metastability) where a diffusion process takes a lot of time to leave
a domain not because of high energy barriers but because of geometric
constraints which make the exit doors very small. In this work, we are interested in
precisely characterizing the law of the exit time and exit point from
a two-dimensional domain, starting from the quasi-stationary
distribution. In particular, the results presented here can be seen as
a counterpart to our results in~\cite{di2016jump,GesGiaLeliPeutrec,LelievrePeutrecNectoux} for energetic
metastability, with the objective to identify the equivalent of
the celebrated Eyring-Kramers laws for entropic barriers.

The geometric setting is as follows. Let us consider a $\mathcal C^\infty$ two-dimensional domain $\Omega
\subset \R^2$,
and $N$ points on the boundary of $\Omega$: $x^{(1)},\dots , x^{(N)}
\in \partial \Omega$. Let us also introduce the parameters encoding the
sizes of the exit windows
$\varepsilon=(\varepsilon_1, \dots , \varepsilon_N) \in \R_+^*$. The
exit region is defined by 
\begin{equation}\label{eq:GammaD}
  \Gamma_{\mathcal D}^\varepsilon=\bigcup_{k=1}^N \Gamma_k^\varepsilon
\text{ where } \Gamma_k^\varepsilon=\partial \Omega \cap B(x^{(k)},
\varepsilon_k)
\end{equation}
where $B(x^{(k)},
\varepsilon_k)$ denotes the open two-dimensional ball centered at
$x^{(k)}$ with radius $\varepsilon_k$.
The subscript ${\mathcal D}$ refers to the fact that
considering exits through $\Gamma_{\mathcal D}$ will amount to imposing
Dirichlet boundary conditions on
$\Gamma_{\mathcal D}^\varepsilon$. Let us also introduce
\begin{equation}\label{eq:GammaN}
\Gamma_{\mathcal N}^\varepsilon = \partial \Omega \setminus
\overline{\Gamma_{\mathcal D}^\varepsilon}
\end{equation}
the reflecting part of the boundary. The subscript ${\mathcal N}$ refers to the
fact that reflecting boundary conditions will translate into Neumann
boundary conditions.

\subsection{Motivation: the probabilistic framework}

The objective of this section is to explain the probabilistic
context. It can be omitted by the reader only interested in the mathematical results.

Let us consider a Brownian motion in $\Omega$ with reflecting
boundary conditions on $\partial \Omega$:
$$dX_t=\sqrt{2} dB_t - 1_{\partial \Omega}(X_t) n_{X_t} dL_t$$
where $X_t \in \Omega$, $(B_t)_{t \ge 0}$ is a two-dimensional
Brownian motion, for any $x \in \partial \Omega$, $n_x$ is the
outward unit normal vector to $\Omega$ and $L_t$ is the local time on
$\partial \Omega$ enforcing the reflecting boundary conditions on
$\partial \Omega$, see for
example~\cite{Freidlin,lions1984stochastic,stroock1971diffusion,watanabe1971stochastic}. Let
us introduce the first time $X_t$ reaches $\Gamma_{\mathcal D}^\varepsilon$:
$$\tau := \inf\{t \ge 0, X_t \in \Gamma_{\mathcal D}^\varepsilon\}.$$
We are interested in precisely identifying the law of the pair of
random variables $(\tau,X_\tau)$ in the regime $|\varepsilon| \to 0$,
where $\varepsilon=(\varepsilon_1, \ldots, \varepsilon_N)$.

Since the Lebesgue measure $|\Gamma_{\mathcal D}^\varepsilon|$ of the exit region
in $\partial \Omega$ goes to zero when $|\varepsilon| \to 0$, it
is natural to assume that $X_t$ reached a local equilibrium within
$\Omega$ before exiting~\cite{di2016jump}. This is formalized by the notion of
quasi-stationary distribution. Indeed, one can show that the law of
$X_t$ conditioned on $t<\tau$  converges when $t \to \infty$ to a probability measure $\nu_0^\varepsilon$
supported in $\Omega$, called the quasi-stationary distribution. In the limit
$|\varepsilon| \to 0$, one expects the stochastic process to reach the
quasi-stationary distribution before exiting: this is a so-called
metastable exit. Moreover, starting from  $\nu_0^\varepsilon$, the random variable $\tau$ is
exponentially distributed and independent of $X_\tau$: this explains
why the quasi-stationary distribution is the appropriate initial
condition to parametrize an underlying jump Markov process to describe
the exit from the state $\Omega$, see~\cite{di2016jump}. Let us insist
on the fact that these two properties hold for every fixed $\varepsilon$. Therefore, in order
to characterize the exit event in such a metastable situation, one needs to identify the
parameter of the exponential law (the inverse of the mean exit time
$\mathbb E_{\nu_0^\varepsilon}(\tau)$), as well as the probabilities to
leave through each of the exit windows $\mathbb P_{\nu_0^\varepsilon}(X_\tau
\in \Gamma^\varepsilon_k)$, for $k \in \llbracket 1, N\rrbracket$. The
exit rate through the exit regions $\Gamma^\varepsilon_k$, for $k
  \in \llbracket 1, N\rrbracket$, is then $\mathbb P_{\nu_0^\varepsilon}(X_\tau
\in \Gamma^\varepsilon_k) \, / \, \mathbb E_{\nu_0^\varepsilon}(\tau)$.
It is not our objective in this work to precisely justify
mathematically the discussion above: we will rather rely on partial
differential equations to introduce these objects, and start from
there to perform the analysis $|\varepsilon| \to 0$. We refer
to~\cite{collet2013quasi,champagnat-villemonais-23} for general
results
about existence, uniqueness and convergence to the quasi-stationary
distribution and also to~\cite[Theorem 1.1, Proposition 1.3]{Louis}
for complete proofs in
our context. 

\subsection{The quantities of interest}\label{sec:qoi}

In order to introduce the quantities of interest in this work let us
relate the probabilistic objects introduced in the previous section with partial differential equations and spectral
quantities, that will be the focus of this work.

Let us denote by $\mathcal L^\varepsilon$ the opposite of the Laplace operator with
Dirichlet boundary conditions on $\Gamma^\varepsilon_{\mathcal D}$ and Neumann
boundary conditions on $\Gamma^\varepsilon_{\mathcal N}$. This operator can be
easily built as the Friedrichs extension of the symmetric quadratic
form $u \in H^1_{0,\Gamma^\varepsilon_{\mathcal D}} \mapsto \int_\Omega |\nabla
u|^2$, where $H^1_{0,\Gamma^\varepsilon_{\mathcal D}} =\{u \in H^1(\Omega), \, u|_{\Gamma^\varepsilon_{\mathcal D}}=0\}$.
It is well known
(see~\cite[Section 2]{LeRaSt}) that $\mathcal L^\varepsilon$ is a
self-adjoint positive operator with compact resolvent. It has
a discrete spectrum, and its smallest eigenvalue
$\lambda_0^\varepsilon >0$ is simple with an associated
eigenvector $u_0^\varepsilon$ which has a sign on $\Omega$:
\begin{equation}\label{eq:u0}
  \left\{
    \begin{aligned}
  - \Delta u_0^\varepsilon &= \lambda_0^\varepsilon u_0^\varepsilon \text{ in } \Omega,\\
  \partial_n u_0^\varepsilon &= 0 \text{ on } \Gamma^\varepsilon_{\mathcal N},\\
  u_0^\varepsilon &= 0 \text{
    on } \Gamma^\varepsilon_{\mathcal D}.
  \end{aligned}
\right.
\end{equation}
In the
following, we assume that $u_0^\varepsilon$ is normalized such that:
$$\int_{\Omega}  u_0^\varepsilon = -\|u_0^\varepsilon\|_{1}=-1.$$
where here and in the following $\|\cdot\|_p$ denotes the
$L^p(\Omega)$ norm.

The density function $f(t,x)$
of the process $X_t$ reflected on $\Gamma_{\mathcal N}^\varepsilon$ and absorbed
on $\Gamma_{\mathcal D}^\varepsilon$ satisfies (see~\cite[Section 2.5]{Freidlin}):
\begin{equation}\label{eq:FP}
  \left\{
    \begin{aligned}
  \partial_t f &= \Delta f \text{ in } \Omega,\\
  \partial_n f &= 0 \text{ on } \Gamma^\varepsilon_{\mathcal N} \text{ and }   f = 0 \text{
    on } \Gamma^\varepsilon_{\mathcal D},\\
  f(0,\cdot)&= f_0 \text{ in } \Omega,
    \end{aligned}
\right.
\end{equation}
where $f_0$ is the density of $X_0$ and $\partial_n= n \cdot \nabla$
denotes the normal derivative on $\partial \Omega$.
 The law of $X_t$
conditioned on $t<\tau$ thus has the density $f(t,x)/ \int_{\Omega}
f(t,x) \, dx$, and, thanks to the properties of $\mathcal
L^\varepsilon$ mentioned above, it converges to $u_0^\varepsilon/\int_\Omega
u_0^\varepsilon$ when $t \to \infty$. In particular, the quasi-stationary distribution is
$$\nu_0^\varepsilon(\D x) = - u_0^\varepsilon(x) 1_\Omega(x) \, \D x.$$
Moreover, it can be shown that, if $X_0$ is distributed according to
$\nu_0^\varepsilon$, then (i) $\tau$ is exponentially distributed
with parameter $\lambda_0^\varepsilon$:
\begin{equation}\label{eq:esp_tau}
  \E_{\nu_0^\varepsilon}(\tau)=\frac{1}{\lambda_0^\varepsilon},
  \end{equation}
(ii) $\tau$ is independent of
$X_\tau$, and (iii) the law of $X_\tau$ is given by: for any smooth
function $\psi: \Gamma^\varepsilon_{\mathcal D} \to \R$,
\begin{equation}\label{eq:exitlaw}
  \E_{\nu_0^\varepsilon} \Big( \psi(X_\tau) \Big) =\frac{1}{\lambda_0^\varepsilon} \left\langle \partial_n u_0^\varepsilon , \psi \right\rangle_{H^{-1/2}, H^{1/2}(\Gamma_{\mathcal D}^\varepsilon)}.
\end{equation}
We refer to
Appendix~\ref{sec:prob_res} for details.

In view of~\eqref{eq:esp_tau} and~\eqref{eq:exitlaw}, in order to
characterize, in the limit $|\varepsilon| \to 0$, the law of the exit event $(\tau,X_\tau)$ for the process
starting under the quasi-stationary distribution, we have to analyze the asymptotic behaviors of $\lambda_0^\varepsilon$ and
$\partial_n u_0^\varepsilon$ in this limiting regime. This is the
objective of this work. It
will be achieved by constructing a very good quasimode for
$u_0^\varepsilon$ as an expansion in powers of
$K^\varepsilon=(K^\varepsilon_1,\ldots,K^\varepsilon_N)$, where
$K^\varepsilon_k=-1/ \ln(\varepsilon_k/2)$, for $k=1,\ldots,N$. 

\subsection{Bibliographic comments and outline}

As already mentioned above, the narrow escape problem has been
extensively considered by physicists, with in particular a
focus on precise estimates in the small $|\varepsilon|$ regime of the mean exit time as a function of the
(deterministic) initial condition
$X_0=x \in \Omega$~\cite{grigoriev2002kinetics,schuss2007narrow,singer2008narrow,SSH,SSH2,bib2,grebenkov2019full,PWK,CWS}. Some
of these results are restricted to specific geometries (disks or
spheres for example). At variance, in our work, we are
interested in both the exit time and the exit point distribution from a
general two-dimensional domain, but we
assume that the initial condition is the quasi-stationary
distribution $\nu_0^\varepsilon$, which, as explained above, is natural when considering metastable exits. This allows us to use variational tools to identify
good approximations of the first eigenfunction $u_0^\varepsilon$.

In the mathematical literature, let us mention that asymptotic
expansions of the mean first exit time have been obtained
in~\cite{Ammari2} using layer potential
techniques~\cite{livreAmmari} (see also~\cite{nursultanov2021mean} for generalizations to
Riemannian manifolds). The quasimode we will build is
partially inspired by techniques from~\cite{livreAmmari}. Besides,
proofs of the result from~\cite{CWS} have been
obtained in~\cite{ChenFriedman}. In particular, to the best of our
knowledge, no mathematical results dealing with the exit point
distributions have been obtained so far. Let us
also mention related works on the convergence of a reflected Brownian
motion on an asymptotically shrinking neighborhood of a graph~\cite{freidlin1993diffusion,hsu2025asymptotic,hsu2025metastability}.

This work can be seen as a continuation of our previous
result~\cite{LeRaSt} where we study a simplified version of this
problem: in~\cite{LeRaSt}, we assume that the domain is a disk, and we modify the
definition of the exit regions (exits through small disks rather than
subsets of $\partial \Omega$) in order to build a quasimode in the
domain ${\mathcal D}({\mathcal L}^\varepsilon)$ of the operator
$\mathcal L^\varepsilon$. The main two novelties of this work are the
following. First,
we can treat the problem as
stated originally, with exit regions which are vanishing subsets of
the boundary of a
fixed domain $\Omega$. From a technical viewpoint, we are able to use a quasimode which is not in
${\mathcal D}({\mathcal L}^\varepsilon)$ (it is actually even not in the
domain of the associated quadratic form). Second, we obtain high order expansions of the quantities of
interest, and not only first order terms.
The companion paper~\cite{Louis} treats the larger dimension cases $d \ge 3$, 
but with exit regions which are not simply subsets of $\partial \Omega$ (the exit regions are essentially balls around $x^{(k)}$, 
and this simplifies the analysis, see also \cite{LeRaSt}). 
Moreover, the expansions obtained in~\cite{Louis} are only performed at first order in the small parameter $|\varepsilon|$.

The outline of this work is as follows. Section~\ref{sec:QM} is
devoted to the construction of a quasimode, namely a precise
approximation of $u_0^\varepsilon$. This quasimode can then be used to
obtain high order asymptotic results on the mean exit time
$\E_{\nu_0^\varepsilon}(\tau)=1/\lambda_0^\varepsilon$ in
Section~\ref{sec:MET} (see Theorem~\ref{lam0}) and finally on the
probabilities $\mathbb P_{\nu_0^\varepsilon}(X_\tau \in
\Gamma^\varepsilon_k)$ ($k \in \llbracket 1, N\rrbracket$) in
Section~\ref{sec:LFEP} (see Theorem~\ref{thm:exit_pt}).


\section{Construction of the quasimode}\label{sec:QM}

\subsection{Preliminaries}

Remember the geometric setting introduced above:
$\Omega$ is a $\mathcal C^\infty$ two-dimensional domain, and its
boundary is partitioned as: $\partial \Omega =
\overline{\Gamma^\varepsilon_{\mathcal D}} \cup \overline{\Gamma^\varepsilon_{\mathcal N}}$
where $\Gamma^\varepsilon_{\mathcal D}$ and $\Gamma^\varepsilon_{\mathcal N}$ are two disjoint open sets, defined by~\eqref{eq:GammaD} and~\eqref{eq:GammaN}.

Let $k\in \llbracket 1, N\rrbracket$. Since the domain is regular, there exist $U_0$ a neighborhood of $0$ in $\R^2$ as well as $U_k$ a neighborhood of $x^{(k)}$ in $\R^2$ and $\Psi_k:U_0 \to U_k$ a smooth diffeomorphism such that $\Psi_k(0)=x^{(k)}$, $D_0 \Psi_k$, 
the Jacobian of $\Psi_k$ at $0$, is unitary,
$$\Omega\cap U_k=\Psi_k\Big( \left( \R_-^* \times \R \right) \cap U_0 \Big) \qquad \text{and} \qquad \Gamma_k^\varepsilon=\Psi_k\left(\{0\}\times \big(T^-_{k,\varepsilon},T^+_{k,\varepsilon}\big)\right)$$
for some real numbers $T^-_{k,\varepsilon}$, $T^+_{k,\varepsilon}$ satisfying
\begin{align}\label{T-T+}
|T^-_{k,\varepsilon}|\leq T^+_{k,\varepsilon}
\end{align}
as well as 
\begin{align}\label{approxT}
T^-_{k,\varepsilon}=-\varepsilon_k+O(|\varepsilon|^2)\qquad \text{and} \qquad T^+_{k,\varepsilon}=\varepsilon_k+O(|\varepsilon|^2)
\end{align}
where here and in the following, $O(\cdot)$ always refer to the limit
$\varepsilon \to 0$.
For $(0,t)\in U_0$, we will use the parametrization of
$\Gamma^\varepsilon_k$: for $t \in (T^-_{k,\varepsilon}, T^+_{k,\varepsilon})$,
\begin{equation}\label{eq:psik}
  \psi_k(t)=\Psi_k(0,t) \in \Gamma^\varepsilon_k.
  \end{equation}

Our objective is to build a quasimode for the first eigenelements of
the operator ${\mathcal L}^\varepsilon$, namely a good approximation
of $\lambda_0^\varepsilon$ and of the function $u_0^\varepsilon$ introduced above
in~\eqref{eq:u0}. The quasimode will be a function satisfying
\begin{enumerate}[label=\textbullet]
\item an approximation in $L^2(\Omega)$ of the eigenvalue problem~\eqref{eq:u0} (see \eqref{vpappro} below)
\item approximate homogeneous Dirichlet boundary conditions in $L^\infty(\Gamma_{\mathcal D}^\varepsilon)$
\item exact homogeneous Neumann boundary conditions on $\Gamma^\varepsilon_{\mathcal N}$.
\end{enumerate}
We will therefore need some precise estimates both in $L^2(\Omega)$
and in $L^\infty(\Gamma_{\mathcal D}^\varepsilon)$. Let us emphasize that because
of the approximate Dirichlet boundary conditions, the quasimode will not belong
to ${\mathcal D}({\mathcal L}^\varepsilon)$.

Let us now informally explain how the quasimode will be built (the
rigorous construction is done in Sections~\ref{sec:R}
and~\ref{sec:S}).
We will need positive coefficients
$K_k^\varepsilon$ which depend on $\varepsilon_k$, and which satisfy
$K_k^\varepsilon \to 0$ as $\varepsilon_k\to 0$. These coefficients
will be determined later on in order for the quasimode to satisfy
(approximately) the
Dirichlet boundary conditions (see Equation~\eqref{defK} below). 
We will denote $K^\varepsilon=(K^\varepsilon_1,\dots , K^\varepsilon_N)$.
For any multi-index $\gamma\in \N^{N}$, we will denote
$$|\gamma|=\sum_{k=1}^{N} \gamma_k$$
its length and
$$(K^\varepsilon)^\gamma=\prod_{k=1}^{N} (K_k^\varepsilon)^{\gamma_k}$$
the associated monomial in the variables $K^\varepsilon$.
Let $M \ge 1$ be an integer, and let us consider a quasimode of the
form: for all $x \in \Omega$, 
\begin{equation}\label{defphi}
    \varphi^\varepsilon(x) = 1+ 
 \mathop{\sum_{\gamma \in \N^N;}}_{\substack{1\leq |\gamma|\leq M}} (K^\varepsilon)^\gamma f_\gamma^\varepsilon(x)  
\end{equation}
with $f_\gamma^\varepsilon\in L^2(\Omega)$.
Notice that we adopt an unusual notion of expansion as we allow the functions $f_\gamma^\varepsilon$ to depend on $\varepsilon$.
However, we will build the functions $f_\gamma^\varepsilon$ so that
\begin{align}\label{limf}
f_\gamma^\varepsilon=f_\gamma^0+O\left(|\varepsilon|^{1/3}\right) \qquad \text{in } L^2(\Omega).
\end{align}

For any choice of real numbers
$(a_\gamma)_{\gamma \in \N^N, |\gamma|\le M}$ (these
  will be carefully chosen later on in order to satisfy approximately
  the Dirichlet boundary conditions), $\varphi^\varepsilon$ satisfies the approximate eigenvalue problem
  \begin{align}\label{vpappro}
    \left\{
    \begin{aligned}
&\bigg( -\Delta -
 \mathop{\sum_{\gamma \in \N^N;}}_{\substack{1\leq |\gamma|\leq M}}
 a_\gamma (K^\varepsilon)^\gamma \bigg) \varphi^\varepsilon=O\left(
   |K_\varepsilon|^{M+1} \right) +O\left(|\varepsilon|^{1/3}\right)
 \qquad \text{in } L^2(\Omega)\\
 &\partial_n \varphi^\varepsilon = 0 \text{ on }
 \Gamma^\varepsilon_{\mathcal N}
\end{aligned}
    \right.
\end{align}
if for all 
$\gamma\in \N^N$ such that $1\leq |\gamma|\leq M$,
\begin{equation}
\left \{
    \begin{aligned}
        \Delta f_\gamma^\varepsilon & = -a_\gamma-F_\gamma
        &&\text{ in }  \Omega,\\
        \partial_n f_\gamma^\varepsilon & = 0 &&\text{ on } \Gamma^\varepsilon_{\gamma,\mathcal N}
    \end{aligned}
    \right . 
\end{equation}
where here and in the following, 
$$  \Gamma^\varepsilon_{\gamma,\mathcal D}=   
\bigcup_{k \in \mathrm{supp}\, \gamma} \Gamma_k^\varepsilon\text{ and }
\Gamma^\varepsilon_{\gamma,\mathcal N}=  \partial \Omega \setminus \overline{
\Gamma^\varepsilon_{\gamma,\mathcal D}}$$
and
\begin{align}\label{eq:F}
F_\gamma=0\qquad  \text{if } |\gamma|=1 \qquad \qquad \text{while} \quad 
F_\gamma=
	 \mathop{\sum_{\beta \in \N^N;\; \beta\leq \gamma;}}_{\substack{1\leq |\beta|\leq |\gamma|-1}} a_{\gamma-\beta} f_\beta^0 
\qquad \text{ if } |\gamma|>1.
\end{align}
The subscript $\beta\leq \gamma$ means that for all $k\in
\llbracket 1,N \rrbracket$, $\beta_k\leq \gamma_k$, so that
$\gamma-\beta\in \N^N$.

In order to be able to choose the parameters $(a_\gamma)$ and
$(K^\varepsilon_k)$ such that the
quasimode $\varphi^\varepsilon$ vanishes (at least approximately)
on $\Gamma_{\mathcal D}^\varepsilon$, in view of the definition of
$\varphi^\varepsilon$ \eqref{defphi}, we would like
$f^\varepsilon_\gamma$ to be approximately constant on each
$(\Gamma_k^\varepsilon)_{k\in \mathrm{supp}\, \gamma}$:
\begin{equation}
\label{eq:fgamepd}
\left \{
    \begin{aligned}
        \Delta f_\gamma^\varepsilon & = -a_\gamma-F_\gamma
        &&\text{ in }  \Omega,\\
        \partial_n f_\gamma^\varepsilon & = 0 &&\text{ on } \Gamma^\varepsilon_{\gamma,\mathcal N},\\
	f_\gamma^\varepsilon &\approx f_\gamma^\varepsilon(x^{(k)})
        &&\text{ on }\Gamma_k^\varepsilon \; \text{for all }k\in \mathrm{supp}\, \gamma.
    \end{aligned}
    \right . 
  \end{equation}
  A precise meaning to this will be given below.
Note that $F_\gamma$ only depends on the functions $f_\beta^0$ and the coefficients $a_{\gamma-\beta}$ where 
$|\beta|$ and $|\gamma-\beta|$ both belong to $\llbracket 1, |\gamma|-1 \rrbracket$.
The equations \eqref{eq:fgamepd} can thus be solved by induction on $M$.

Let us first consider $M=1$. When $M=1$, the multi-indices $\gamma$ are of the form $\gamma=e_k=(\delta_{k,j})_{1\leq j\leq N}$ for some $k\in \llbracket 1,N \rrbracket$ and we simply denote $f_\gamma^\varepsilon=f_k^\varepsilon$ and $a_\gamma=a_k$.
Equation \eqref{eq:fgamepd} becomes (since $F_\gamma=0$ if $|\gamma|=1$)
\begin{equation}\label{eq:fkepd}
\left \{
    \begin{aligned}
        \Delta f_k^\varepsilon & = -a_k &&\text{ in }  \Omega,\\
        \partial_n f_k^\varepsilon & = 0 &&\text{ on } \partial \Omega \backslash  \overline{\Gamma_k^\varepsilon}\\
	f_k^\varepsilon &\approx f_k^\varepsilon(x^{(k)}) &&\text{ on }  \Gamma_k^\varepsilon.
    \end{aligned}
    \right . 
\end{equation}
Let us explain how the functions $f_k^\varepsilon$
solution to \eqref{eq:fkepd} are built.
For $k\in \llbracket  1,N \rrbracket$ and $t\in [T^-_{k,\varepsilon}\,
, \, T^+_{k,\varepsilon}]$, let us denote
\begin{equation}\label{eq:wk}
  w_k^\varepsilon(t)=\chi_k^\varepsilon(t)\tilde w_k^\varepsilon(t)
  \end{equation}
where
\begin{equation}\label{eq:tildewk}
\tilde w_k^\varepsilon(t)=\left(1-\left(
    \frac{t}{T^+_{k,\varepsilon}} \right)^2\right)^{-1/2}
\end{equation}
and $\chi_k^\varepsilon$ is a smooth function taking values in $[0,1]$ and satisfying
\begin{align}\label{propchi}
\chi_k^\varepsilon\equiv 1 \text{ on } \left[T^-_{k,\varepsilon}+\varepsilon_k^2\, , \, |T^-_{k,\varepsilon}|-\varepsilon_k^2\right]\qquad \text{as well as} \qquad \mathrm{supp} \, \chi_k^\varepsilon \subset \left[T^-_{k,\varepsilon}+\varepsilon_k^2/2\, , \, |T^-_{k,\varepsilon}|-\varepsilon_k^2/2\right].
\end{align}
Note that thanks to \eqref{T-T+} and \eqref{propchi}, we have $w_k^\varepsilon\in L^\infty(
[T^-_{k,\varepsilon}\, , \, T^+_{k,\varepsilon}
]).$
The function $f^\varepsilon_k$ will be defined as:
\begin{equation}\label{deff}
    f_k^\varepsilon(x) = S_k^\varepsilon(x) + R_k^\varepsilon(x)  \in L^2(\Omega)
\end{equation}
where 
\begin{equation}\label{defS}
S_k^\varepsilon(x)=\frac{1}{\|w_k^\varepsilon\|_1}\int_{T^-_{k,\varepsilon}}^{T^+_{k,\varepsilon}} w_k^\varepsilon(t) \ln \Big(|x-\psi_k(t)|\Big) \D t,
\end{equation}
\begin{equation}\label{defR}
R_k^\varepsilon(x)=\frac{1}{\|w_k^\varepsilon\|_1}\int_{T^-_{k,\varepsilon}}^{T^+_{k,\varepsilon}} w_k^\varepsilon(t) \tilde R_k(x,t) \D t
\end{equation}
for some function $\tilde R_k(x,t)$ 
to be determined.
Here the notations $S$ and $R$ stand respectively for \emph{singular} and \emph{regular}.
We will in particular see in section \ref{sec:R} that there exists
some functions $\tilde R_k(x,t)$ such that
$R_k^\varepsilon(x)$ is Hölder continuous and $f_k^\varepsilon$ solves
the first two lines of \eqref{eq:fkepd}.
The last line of \eqref{eq:fkepd} will be satisfied thanks to the choice of the weight $w_k^\varepsilon$ (see Lemma \ref{Sbord} below) and the regularity of $R_k^\varepsilon$.
We will also see in Lemma \ref{f0} that the obtained function $f_k^\varepsilon$ will indeed satisfy \eqref{limf}.

Let us now consider $M>1$ and suppose that the functions
$f^\varepsilon_\gamma$ have been determined for all $\gamma$ such that
$1\leq |\gamma|\leq M-1$. Notice that from~\eqref{eq:F}, $F_\gamma$ is
fully determined for any $\gamma$ such that $|\gamma|=M$.
Let $v_\gamma$ be the solution of
\begin{equation}\label{vgam}
\left \{
    \begin{aligned}
        \Delta v_\gamma & = -F_\gamma + \langle F_\gamma \rangle
        &&\text{ in }  \Omega,\\
        \partial_n v_\gamma & = 0 &&\text{ on } \partial \Omega, \\
\langle v_\gamma \rangle&=0
    \end{aligned}
    \right . 
  \end{equation}
  where here and in the following $\langle \cdot \rangle$ denotes the
  averages over $\Omega$: $\langle F_\gamma \rangle =
  \frac{1}{|\Omega|}\int_\Omega F_\gamma$. 
Since $F_\gamma \in L^2(\Omega)$, one has by elliptic
regularity~\cite[Theorem 9.26]{brezis2011functional} and Sobolev injection
\begin{align}\label{uholder}
v_\gamma \in H^2(\Omega) \qquad \text{and thus}\qquad  v_\gamma \in \mathcal C^\alpha(\overline \Omega) \qquad \text{for all }\alpha\in [0,1).
\end{align}
Thanks to \eqref{eq:fkepd}, for any choice of the
real coefficients $(b_{\gamma,k})_{k\in \mathrm{supp}\, \gamma}$, the function 
\begin{align}\label{fgamma}
f_\gamma^\varepsilon=v_\gamma+\sum_{k\in \mathrm{supp}\, \gamma}  b_{\gamma,k}f_k^\varepsilon
\end{align}
satisfies~\eqref{eq:fgamepd} with $a_\gamma$ defined by:
$$
a_\gamma=-\langle F_\gamma \rangle +\sum_{k\in \mathrm{supp}\, \gamma}
a_k b_{\gamma,k}.
$$
Moreover, as soon as the functions $(f_k^\varepsilon)_{k \in
  \llbracket 1,N\rrbracket}$ satisfy~\eqref{limf}, so do the functions
$f_\gamma^\varepsilon$. 
Note that thanks to~\eqref{eq:F}, the definition~\eqref{fgamma} of
$f_\gamma^\varepsilon$ can be extended to the case $|\gamma|=1$,
i.e. $\gamma=e_k$ for some $k \in \llbracket 1,N\rrbracket$, in order to coincide with $f_k^\varepsilon$ by choosing
\begin{align}\label{c1}
\text{ for any } \gamma \text{ such that } |\gamma|=1, \, v_\gamma=0
  \text{ and } b_{\gamma,k}=1.
\end{align}

At this stage, we have proven the following proposition which already
shows how to build $\varphi^\varepsilon$ such that~\eqref{vpappro} holds.
\begin{prop}\label{blibre}
  Let us introduce two sets of real coefficients: $(a_k)_{k\in
    \llbracket 1, N \rrbracket}$, and for any $\gamma$ such that $|\gamma|>1$, $(b_{\gamma,k})_{k\in \mathrm{supp}\, \gamma}$.
Suppose that for all $k\in \llbracket 1, N \rrbracket$, the function
$f^\varepsilon_k$ satisfies the first two lines of \eqref{eq:fkepd},
as well as~\eqref{limf}.
Suppose that for all $\gamma$ such that $|\gamma|>1$, the functions
$f_\gamma^\varepsilon$ are defined recursively by~\eqref{fgamma} (in
particular, they all satisfy~\eqref{limf}).
Then the function $\varphi^\varepsilon$ defined by \eqref{defphi} satisfies the approximate eigenvalue problem \eqref{vpappro}, with
\begin{align}\label{compgamma}
a_\gamma=-\langle F_\gamma \rangle +\sum_{k\in \mathrm{supp}\, \gamma} a_k b_{\gamma,k}
\end{align}
for all $\gamma$ such that $2\leq |\gamma|\leq M$. In particular, one has
$\partial_n \varphi^\varepsilon=0$ on $\Gamma^\varepsilon_{\mathcal N}$.
\end{prop}
\hop
The values of $(a_k)_{k\in \llbracket 1, N \rrbracket}$ will be fixed
in Lemma~\ref{lemcomp} below by a simple compatibility condition,
required to build the functions $(f^\varepsilon_k)_{k\in \llbracket 1, N \rrbracket}$. The other coefficients will be determined using a recursion on
  $|\gamma|=m$: for $|\gamma|=m=1$, the coefficients $(K^\varepsilon_k)_{k
  \in \llbracket 1, N\rrbracket}$ will be chosen so that the function
$\varphi^\varepsilon$ truncated at order $M=1$ (see~\eqref{defphi}) approximately satisfies the
homogeneous Dirichlet boundary conditions on
$\Gamma_{\mathcal D}^\varepsilon$ at order~$1$ (i.e. its $L^\infty$ norm on $\Gamma^\varepsilon_{\mathcal D}$
is $O(|K^\varepsilon|)$, see~\eqref{defK}). For $|\gamma|=m\ge 1$,
assuming that, for all $\beta$ such that $|\beta| <m$, the
coefficients  $(b_{\beta,k})_{k\in \mathrm{supp}\,
  \beta}$ and $a_\beta$ 
have been determined, one can build the functions~$F_\gamma$
(see~\eqref{eq:F}), $v_\gamma$ (see~\eqref{vgam}) and thus,
for any choice of the  coefficients  $(b_{\gamma,k})_{k\in \mathrm{supp}\,
  \gamma}$, the coefficients
$a_\gamma$ and the associated
functions $f^\varepsilon_\gamma$ (see~\eqref{fgamma}) which
satisfy~\eqref{eq:fgamepd}. The coefficients  $(b_{\gamma,k})_{k\in \mathrm{supp}\,
  \gamma}$ will then be chosen such that the function
$\varphi^\varepsilon$ truncated at order $M=m$ (see~\eqref{defphi}) approximately satisfies the
homogeneous Dirichlet boundary conditions on
$\Gamma_{\mathcal D}^\varepsilon$ at order $m$ (its $L^\infty$ norm on
$\Gamma^\varepsilon_{\mathcal D}$ is $O(|K^\varepsilon|^m)$, see
Lemmas~\ref{coeffcercle} and~\ref{resolcercle}). This
recursive procedure is iterated until $|\gamma|=M$. The following
sections will provide details about this recursive construction.

\subsection{Study of $R_k^\varepsilon$}\label{sec:R}

The objective of this section is to prove results on the function
$R_k^\varepsilon$ defined by~\eqref{defR}, and in particular to
precisely define the function $\tilde R_k$.
Since for all $t\in [T^-_{k,\varepsilon},T^+_{k,\varepsilon}]$ it
holds (remember the definition~\eqref{eq:psik} of $\psi_k$, and that
$n_x$ denotes the outward unit normal to $\Omega$ at $x \in \partial \Omega$)
\begin{equation}\label{eq:log}
\left \{
    \begin{aligned}
        \Delta_x \left( \ln \Big(|x-\psi_k(t)|\Big) \right)  & = 0
        &&\text{ in }  \Omega,\\
        \partial_{n_x} \left( \ln \Big(|x-\psi_k(t)|\Big) \right) & = N_k(x,t) &&\text{ on } \partial \Omega \backslash  \{\psi_k(t)\}
    \end{aligned}
    \right . 
\end{equation}
where 
$$N_k(x,t)=\frac{n_x\cdot (x-\psi_k(t))}{|x-\psi_k(t)|^2},$$
we are looking in view of \eqref{eq:fkepd}, \eqref{deff}, \eqref{defS}
and \eqref{defR} for a function $\tilde R_k(x,t)$
satisfying
\begin{equation}\label{eq:Rkepd}
\forall t \in [T^-_{k,\varepsilon},T^+_{k,\varepsilon}], \quad \left \{
    \begin{aligned}
        \Delta_x \tilde R_k(x,t)  & = -a_k
        &&
 x\in \Omega, \\
        \partial_{n_x} \tilde R_k(x,t) & = -N_k(x,t) &&
 x\in \partial \Omega \backslash  \{\psi_k(t)\}.
    \end{aligned}
    \right . 
  \end{equation}
Notice that the function $\tilde R_k$ only depends on $\varepsilon$
through the domain of the parameter $t$.

The function $N_k(\cdot,t)$ is a bounded function over the
whole of $\partial \Omega$ as shown in the next result.
\begin{lem}\label{Ninfty}
There exists $C>0$ such that for all $\varepsilon\in (\R_+^*)^N$ and $t\in [T^-_{k,\varepsilon},T^+_{k,\varepsilon}]$, the function $N_k(\cdot,t)$ belongs to $L^\infty(\partial \Omega)$ and
$$\|N_k(\cdot,t)\|_\infty\leq C.$$
\end{lem}

\begin{proof}
We already have
$$\sup_{(x,t)\in \partial \Omega\backslash U_k \times [T^-_{k,\varepsilon},T^+_{k,\varepsilon}] } |N_k(x,t)|\leq C$$
since $N_k$ is smooth on $\partial \Omega\backslash U_k \times [T^-_{k,\varepsilon},T^+_{k,\varepsilon}]$.
Now, for $x=\psi_k(\theta)\in  U_k \cap \partial \Omega$ (with $(0,\theta) \in U_0$) and $t\in [T^-_{k,\varepsilon},T^+_{k,\varepsilon}]$, we have
$$x-\psi_k(t)=\psi_k(\theta)-\psi_k(t)=\psi_k'(\theta)(\theta-t)+O\left((\theta-t)^2\right).$$
Since $\psi_k'(\theta)\perp n_x$ and $|\psi_k'(\theta)|\geq \frac1C$
(for a constant $C>0$ independent of $x \in U_k \cap \partial \Omega$), this gives
$$|N_k(x,t)| =  \frac{O\left((\theta-t)^2\right)}{|\psi_k'(\theta)|^2 (\theta-t)^2 +O\left( (\theta-t)^3 \right) }\leq C^2$$
up to reducing $U_0$.
This proves the statement.
\end{proof}
Thanks to Lemma \ref{Ninfty}, the system \eqref{eq:Rkepd} can be replaced by
\begin{equation}\label{eq:Rkepd2}
  \forall t \in [T^-_{k,\varepsilon},T^+_{k,\varepsilon}], \qquad \left \{
    \begin{aligned}
        \Delta_x \tilde R_k(x,t)  & = -a_k
        &&\text{ in } \Omega, \\
        \partial_{n_x} \tilde R_k(x,t) & = -N_k(x,t) &&\text{ on } 
\partial \Omega.
    \end{aligned}
    \right . 
\end{equation}

\begin{lem}\label{lemcomp}
The compatibility condition from \eqref{eq:Rkepd2} yields
$$a_k =\frac{\pi}{|\Omega|}.$$
\end{lem}
Combined with \eqref{compgamma} and \eqref{c1}, this lemma confirms that for all $1\leq |\gamma|\leq M$, the choice of $a_\gamma$ is entirely determined by the choice of $(b_{\gamma,k})_{k\in \mathrm{supp}\, \gamma}$.

\begin{proof}
  Let $t\in [T^-_{k,\varepsilon},T^+_{k,\varepsilon}]$ be fixed and let us denote $I(t)= \int_{\partial \Omega} N_k(x,t) d\sigma(x)$.
 For
  the partial differential equation~\eqref{eq:Rkepd2} to admit a
  solution, one needs the compatibility condition:
  $$-a_k |\Omega| = \int_\Omega \Delta_x \tilde R_k(x,t)
  \, dx = \int_{\partial \Omega} \partial_{n_x} \tilde
  R^\varepsilon_k(x,t) \sigma(dx) = -I(t).$$
  We will show that whatever $k \in \llbracket 1,N\rrbracket$ and $t \in
  [T^-_{k,\varepsilon},T^+_{k,\varepsilon}]$, $I(t)=\pi$ and this
  will conclude the proof.
 
For all $\delta >0$, the partition $\Omega=\left(\Omega \cap B(\psi_k(t),\delta)\right) \sqcup \left( \Omega \backslash B(\psi_k(t), \delta) \right) $ yields
\begin{align}\label{I}
I(t)=I_1(t)-I_2(t)+I_3(t)
\end{align}
with
$$I_1(t)=\int_{\partial \left(\Omega\backslash B(\psi_k(t),\delta)\right)} N_k(x,t) \,   \D \sigma(x),$$
$$I_2(t)=\int_{\mathcal C(\psi_k(t),\delta)\cap \overline{\Omega}}  N_k(x,t) \, \D \sigma(x),$$
and
$$I_3(t)=\int_{\partial \Omega \cap B(\psi_k(t),\delta)} N_k(x,t) \,
\D \sigma(x),$$
where in $I_1$ and $I_2$, the vector $n_x$ for $x \in \mathcal
C(\psi_k(t),\delta)\cap \overline{\Omega}$ (which appears in the
definition of $N_k$) is defined as the unit
outward normal to $\Omega\backslash B(\psi_k(t),\delta)$.
Now, Green's formula and \eqref{eq:log} yield 
\begin{align}\label{I1}
I_1(t)=\int_{ \Omega\backslash B(\psi_k(t),\delta)} \Delta \ln \left| x - \psi_k(t)\right| \,   \D x=0.
\end{align}
For the computation of $I_2(t)$, let us study $\mathcal C(\psi_k(t),\delta)\cap \overline{\Omega}$ more carefully.
There exist 
\begin{align}\label{y+-}
y_-=t-\frac{\delta}{|\psi_k'(t)|}+O(\delta^2)\qquad \text{and} \qquad   y_+=t+\frac{\delta}{|\psi_k'(t)|}+O(\delta^2)
\end{align}
such that $\mathcal C(\psi_k(t),\delta)\cap \partial \Omega=\{\psi_k(y_-),\psi_k(y_+)\}$.
Besides, there exist $\theta_-\in \R$ and $\theta_+\in [\theta_-,\theta_-+2\pi]$ such that
$$\mathcal C(\psi_k(t),\delta)\cap \overline{\Omega}=\{\psi_k(t)+\delta(\cos \theta, \sin \theta)\, ; \, \theta \in [\theta_-, \theta_+] \}$$
We therefore have
$$\left(\psi_k(y_-)-\psi_k(t)\right)\cdot \left(\psi_k(y_+)-\psi_k(t)\right)=\delta^2 \cos \left(\theta_+-\theta_-\right).$$
On the other hand, we have by a Taylor expansion and using \eqref{y+-}
$$\left(\psi(y_-)-\psi_k(t)\right)\cdot \left(\psi(y_+)-\psi_k(t)\right)=|\psi_k'(t)|^2(y_--t)(y_+-t) +O(\delta^3)
=\delta^2 \big( -1+O(\delta) \big).$$
We deduce that $\cos\left(\theta_+-\theta_-\right)=-1+O(\delta)$, i.e 
\begin{align}\label{theta-theta}
\theta_+-\theta_-=\pi+O(\sqrt \delta).
\end{align}
Now, since $n_x=\delta^{-1}(\psi_k(t)-x)$ on $\mathcal C(\psi_k(t),\delta)$, we obtain using polar coordinates centered at $\psi_k(t)$ that
\begin{align}\label{I2}
I_2(t)=-\int_{\mathcal C(\psi_k(t),\delta)\cap \overline{\Omega}}  \frac{\delta^{-1}\delta^2}{\delta^2} \, \D \sigma(x)=-\int_{\theta_-}^{\theta_+} \D \theta=-\pi+O(\sqrt\delta)
\end{align}
thanks to \eqref{theta-theta}.
Finally, Lemma \ref{Ninfty} yields the following estimate of $I_3(t)$:
\begin{align}\label{I3}
|I_3(t)|\leq \int_{\partial \Omega \cap B(\psi_k(t),\delta)} \|N_k(.,t)\|_\infty   \,  \D \sigma(x)\leq C \sigma\left( \partial \Omega \cap B(\psi_k(t),\delta) \right)\leq C \delta.
\end{align}
Combining \eqref{I}, \eqref{I1}, \eqref{I2} as well as \eqref{I3} and
taking the limit $\delta \to 0$ yields $I(t)=\pi$ which concludes the proof.
\end{proof}



The following is a direct consequence of \cite[Proposition 2.3]{Aramaki} (we restrict ourselves to $p>2$ so
  that $W^{1,p}(\Omega)$ is continuously embedded into ${\mathcal
    C}^0(\overline \Omega)$, and thus $\tilde
  R_k(x^{(k)},t )$ is well defined).
\begin{prop}\label{solR}
For all $t\in [T^-_{k,\varepsilon},T^+_{k,\varepsilon}]$ and any $p
\in (2, \infty)$, there exists
a unique solution $\tilde R_k(\cdot,t) \in W^{1,p}(\Omega)$ to
\eqref{eq:Rkepd2} satisfying $\tilde R_k(x^{(k)},t ) = 0$.
Moreover, for all $p \in (2,\infty)$, there exists $C>0$ such that for all $t\in [T^-_{k,\varepsilon},T^+_{k,\varepsilon}]$, it holds
$$\|\tilde R_k(\cdot,t) \|_{W^{1,p}(\Omega)}\leq C.$$
\end{prop}

\hip
The notation $\tilde R_k$ will now refer to the solution of \eqref{eq:Rkepd2} given by Proposition \ref{solR}.

\begin{cor}\label{Rhold}
The function $R_k^\varepsilon $ defined in \eqref{defR} satisfies 
\begin{equation}
\left \{
    \begin{aligned}
        \Delta  R_k^\varepsilon  & = -a_k
        &&\text{ on }
 \Omega, \\
        \partial_{n}  R_k^\varepsilon & = \frac{-1}{\|w_k^\varepsilon\|_1}\int_{T^-_{k,\varepsilon}}^{T^+_{k,\varepsilon}}  w_k^\varepsilon(t)  N_k(x,t) \D t &&\text{ on } 
\partial \Omega \backslash  \Gamma^\varepsilon_k
    \end{aligned}
    \right . 
\end{equation}
as well as $R_k^\varepsilon(x^{(k)})=0$.
Moreover, $R_k^\varepsilon$ belongs to $W^{1,3}(\Omega)$ (and thus to $\mathcal C^{1/3}(\overline{\Omega})$) with
$$\|R_k^\varepsilon \|_{W^{1,3}(\Omega)}\leq C \qquad \text{and} \qquad 
\|R_k^\varepsilon \|_{\mathcal C^{1/3}(\overline{\Omega})}\leq C.$$
\end{cor}

\begin{proof}
The first statement is a direct consequence of Proposition \ref{solR}.
For the second statement, we have using Proposition \ref{solR} again as well as Minkowski's integral inequality
\begin{align}\label{RW1}
\| R_k^\varepsilon \|_{W^{1,3}(\Omega)}\leq C
\end{align}
and the conclusion follows by the Sobolev injection $W^{1,3}(\Omega) \subset \mathcal C^{1/3}(\overline{\Omega})$ (see for example \cite[Theorem 6 p.270]{Evans}).
\end{proof}

\subsection{Study of $S_k^\varepsilon$}\label{sec:S}

The objective of this section is to prove results on the function
$S_k^\varepsilon$ defined by~\eqref{defS}.

\begin{lem}\label{norm1}
Recall the definitions~\eqref{eq:wk} and ~\eqref{eq:tildewk} of
$w_k^\varepsilon $ and $\tilde w_k^\varepsilon $. We have
$$\|\tilde w_k^\varepsilon\|_1=\pi \varepsilon_k\left(1+O\left(|\varepsilon|\right)\right)$$
and
$$\|w_k^\varepsilon\|_1=\pi \varepsilon_k\left(1+O\left(|\varepsilon|^{1/2}\right)\right).$$
\end{lem}

\begin{proof}
By using the change of variables $z=\frac{t}{T^+_{k,\varepsilon}}$,
one computes (using~\eqref{approxT})
$$\|\tilde w_k^\varepsilon\|_1=\int_{-T^+_{k,\varepsilon}}^{T^+_{k,\varepsilon}} \tilde w_k^\varepsilon (t)  \D t= T^+_{k,\varepsilon} \int_{-1}^{1}   \left(1-z^2\right)^{-1/2}  \D z=T^+_{k,\varepsilon} \left[ \arcsin z \right]_{-1}^{1} =\pi \varepsilon_k+O(|\varepsilon|^2).$$
Besides, \eqref{propchi} and~\eqref{T-T+} allows us to write
\begin{align}\label{decompow}
\|w_k^\varepsilon\|_1=\int_{-T^+_{k,\varepsilon}}^{T^+_{k,\varepsilon}} \tilde w_k^\varepsilon (t)  \D t + \int_{-T^+_{k,\varepsilon}}^{T^+_{k,\varepsilon}} \left(\chi_k^\varepsilon(t)-1\right) \tilde w_k^\varepsilon (t)  \D t .
\end{align}
It thus remains to show that the second term from \eqref{decompow} is
of order $O(|\varepsilon_k| |\varepsilon|^{1/2})$.
To do so, we can use \eqref{propchi} and \eqref{approxT} again as well as the same change of variables to write
\begin{align}
\int_{-T^+_{k,\varepsilon}}^{T^+_{k,\varepsilon}} \left(\chi_k^\varepsilon(t)-1\right) \tilde w_k^\varepsilon (t)  \D t &\leq C \int_{T^+_{k,\varepsilon}-C \varepsilon_k^2}^{T^+_{k,\varepsilon}} 
 \tilde w_k^\varepsilon (t)  \D t \leq C T^+_{k,\varepsilon} \int_{1-C\varepsilon_k}^{1}  \left(1-z^2\right)^{-1/2}  \D z\\
		&\leq C \varepsilon_k \int_{1-C\varepsilon_k}^{1}
           \left(1-z\right)^{-1/2}  \D z \leq C \varepsilon_k \left[- \left(1-z\right)^{1/2} \right]_{1-C\varepsilon_k}^{1}\leq C \varepsilon_k^{3/2}
\end{align}
which concludes the proof.
\end{proof}

\begin{lem}\label{Sbord}
For all $k\in \llbracket 1,N \rrbracket$, the function $S_k^\varepsilon$ 
satisfies
\begin{align}\label{f=ln}
S_k^\varepsilon(x)=s_k^\varepsilon+O\left(|\varepsilon|^{1/6}\right)
\end{align}
uniformly in $x\in \Gamma^\varepsilon_k$, where
$$s_k^\varepsilon=\ln \left(\frac{\varepsilon_k}{2}\right).$$
\end{lem}
\hip
It will appear in the proof that the remainder term $O(|\varepsilon|^{1/6})$ is certainly not optimal but it is sufficient for our purpose.

\begin{proof}
For $x=\psi_k(\theta)\in \Gamma_k^\varepsilon$ (with $\theta \in [T^-_{k,\varepsilon},T^+_{k,\varepsilon}]$), we have using a Taylor expansion and the fact that $D_0\Psi_k$ is unitary
\begin{align}
S_k^\varepsilon(x)&=\frac{1}{2\|w_k^\varepsilon\|_1}\int_{T^-_{k,\varepsilon}}^{T^+_{k,\varepsilon}}  w_k^\varepsilon(t) \ln \Big(|\psi_k(\theta)-\psi_k(t)|^2\Big) \D t\\
		&=\frac{1}{2\|w_k^\varepsilon\|_1}\int_{T^-_{k,\varepsilon}}^{T^+_{k,\varepsilon}} w_k^\varepsilon(t) \ln \Big((t-\theta)^2 \left(1+O(\varepsilon_k) \right)\Big) \D t\\
		&=\frac{1}{2\|w_k^\varepsilon\|_1}\int_{T^-_{k,\varepsilon}}^{T^+_{k,\varepsilon}} w_k^\varepsilon(t) \ln \Big((t-\theta)^2 \Big) \D t +O(\varepsilon_k)\\
		&=\frac{1}{\|w_k^\varepsilon\|_1}\int_{-T^+_{k,\varepsilon}}^{T^+_{k,\varepsilon}} \tilde w_k^\varepsilon(t) \ln \Big(|t-\theta| \Big) \D t+\frac{1}{2\|w_k^\varepsilon\|_1}\int_{-T^+_{k,\varepsilon}}^{T^+_{k,\varepsilon}}  \left(\chi_k^\varepsilon(t)-1\right)  \tilde w_k^\varepsilon(t) \ln \Big((t-\theta)^2 \Big) \D t +O(\varepsilon_k). \label{S2termes}
\end{align}
Let us study more carefully the last integral.
Thanks to \eqref{propchi} and Lemma \ref{norm1}, it satisfies
\begin{align}
\frac{1}{2\|w_k^\varepsilon\|_1}&\left|\int_{-T^+_{k,\varepsilon}}^{T^+_{k,\varepsilon}}  \left(\chi_k^\varepsilon(t)-1\right)  \tilde w_k^\varepsilon(t) \ln \Big((t-\theta)^2 \Big) \D t\right|\leq C \varepsilon_k^{-1}\int_{T^+_{k,\varepsilon}-C\varepsilon_k^2}^{T^+_{k,\varepsilon}}    \tilde w_k^\varepsilon(t) \left|\ln \Big((t-|\theta|)^2 \Big)\right| \D t\\
		&\leq C \varepsilon_k^{-1}  \left( \int_{T^+_{k,\varepsilon}-C\varepsilon_k^2}^{T^+_{k,\varepsilon}}    \tilde w_k^\varepsilon(t)^{3/2}  \D t \right)^{2/3}  \left(
           \int_{T^+_{k,\varepsilon}-C\varepsilon_k^2}^{T^+_{k,\varepsilon}}
           \left| \ln \Big((t-|\theta|)^2 \Big) \right|^3 \D t
           \right)^{1/3} \label{holder}
\end{align}
where we used Hölder's inequality.
Proceeding as in the proof of Lemma \ref{norm1}, one can check that
\begin{align}\label{5/4}
\int_{T^+_{k,\varepsilon}-C\varepsilon_k^2}^{T^+_{k,\varepsilon}}    \tilde w_k^\varepsilon(t)^{3/2}  \D t \leq C \varepsilon_k^{5/4} .
\end{align}
Besides, the change of variables $z=t-|\theta|$ and a quick optimization with respect to $\theta$ yield
\begin{align}\label{1/3}
\int_{T^+_{k,\varepsilon}-C\varepsilon_k^2}^{T^+_{k,\varepsilon}}  \left| \ln \Big((t-|\theta|)^2 \Big) \right|^3 \D t&=\int_{T^+_{k,\varepsilon}-|\theta|-C\varepsilon_k^2}^{T^+_{k,\varepsilon}-|\theta|}  \left| \ln \left(z^2 \right) \right|^3 \D z\\
		&\leq \int_{-\frac C2\varepsilon_k^2}^{\frac C2 \varepsilon_k^2}  \left| \ln \left(z^2 \right) \right|^3 \D z\\
		&\leq C \int_{0}^{\frac C2 \varepsilon_k^2} z^{-1/2} \D z\leq C \varepsilon_k.
\end{align}
Putting \eqref{S2termes}, \eqref{holder}, \eqref{5/4} and \eqref{1/3} together shows that
$$S_k^\varepsilon(x)=\frac{1}{\|w_k^\varepsilon\|_1}\int_{-T^+_{k,\varepsilon}}^{T^+_{k,\varepsilon}} \tilde w_k^\varepsilon(t) \ln \Big(|t-\theta| \Big) \D t +O\left(\varepsilon_k^{1/6}\right).$$
It remains to study the first term.
We will follow \cite[section 5.2.3]{livreAmmari} to do so.
It is known (see \cite[p.30]{Muskhelishvili}) that the function
$$\theta \mapsto \frac{1}{\|w_k^\varepsilon\|_1}\int_{-T^+_{k,\varepsilon}}^{T^+_{k,\varepsilon}} \tilde w_k^\varepsilon(t) \ln \Big(|t-\theta| \Big) \D t$$
is differentiable and that
$$  \frac{1}{\|w_k^\varepsilon\|_1} \partial_\theta \int_{-T^+_{k,\varepsilon}}^{T^+_{k,\varepsilon}} \tilde w_k^\varepsilon(t) \ln \Big(|t-\theta| \Big) \D t= \frac{1}{\|w_k^\varepsilon\|_1} \mathrm{v.p} \int_{-T^+_{k,\varepsilon}}^{T^+_{k,\varepsilon}} \frac{ \tilde w_k^\varepsilon(t)}{t-\theta}  \D t=0$$
(see \cite[p.174]{Tricomi} for a proof of the last equality, which is
a well-known property for Hilbert transforms).
It is therefore constant and equal to its value at $\theta=0$ which is 
\begin{align}
\frac{1}{\|w_k^\varepsilon\|_1}\int_{-T^+_{k,\varepsilon}}^{T^+_{k,\varepsilon}} \tilde w_k^\varepsilon(t) \ln \Big(|t| \Big) \D t&=\frac{2}{\|w_k^\varepsilon\|_1}\int_{0}^{T^+_{k,\varepsilon}} \tilde w_k^\varepsilon(t) \left(  \ln T^+_{k,\varepsilon}+ \ln 
\left(\frac{t}{T^+_{k,\varepsilon}}\right) 
  \right) \D t\\
			&=(\ln T^+_{k,\varepsilon}) \frac{\|\tilde w_k^\varepsilon\|_1}{\|w_k^\varepsilon\|_1} + \frac{2T^+_{k,\varepsilon}}{\|w_k^\varepsilon\|_1}\int_{0}^{1} \frac{\ln z}{\sqrt{1-z^2}}
   \D z\\
		&= \left( \ln \varepsilon_k +O(|\varepsilon|)  \right)  \left( 1+O(|\varepsilon|^{1/2}) \right) + \left( \frac{1}{\pi} +O(|\varepsilon|^{1/2})  \right) \int_{-1}^{1} \frac{\ln |z|}{\sqrt{1-z^2}} \D z\\
		&=\ln \varepsilon_k + \int_{-1}^{1} \frac{\ln |z|}{\pi \sqrt{1-z^2}} \D z +O(|\varepsilon|^{1/3})
\end{align}
thanks to \eqref{approxT}, Lemma \ref{norm1} and the change of
variables $z=t/T^+_{k,\varepsilon}$. It remains to show that $I=\int_{-1}^{1} \frac{\ln |z|}{\pi
  \sqrt{1-z^2}} \D z =-\ln 2$ to complete the proof.
One has $I=(2/\pi) \int_0^1 (\ln z) / \sqrt{1-z^2} \D z$ and the change of variable
  $z=\sin(t)$ yields  $I=(2/\pi) \int_0^{\pi/2} \ln( \sin(t)) \D t =
  (2/\pi) \times (-\pi/2) \ln(2)$ (this is a standard Euler integral).
This completes the proof.
\end{proof}

\begin{lem}\label{SH1}
For all $k\in \llbracket 1,N \rrbracket$ and $\varepsilon\in (\R_+^*)^N$, 
the function $S_k^\varepsilon$ belongs to $H^1(\Omega)$.
\end{lem}
\hip
Note that this Lemma does not provide an estimate of
$\|S_k^\varepsilon \|_{H^1(\Omega)}$: it will be used to justify
integrations by parts in Green's formulas.
\hip
\begin{proof}
 The fact that $S_k^\varepsilon \in L^2(\Omega)$ is a direct
   consequence of the Minkowski's integral inequality, and the fact
   that $x\mapsto \ln |x-\psi_k(t)|$ is in $L^2(\Omega)$.
We would like now to prove that $\nabla S_k^\varepsilon\in L^2(\Omega)$.
For $x\in \Omega$, we have by differentiating under the integral
\begin{align}\label{nablaS}
\nabla  S_k^\varepsilon  (x)=\frac{1}{\|w_k^\varepsilon\|_1}\int_{T^-_{k,\varepsilon}}^{T^+_{k,\varepsilon}}  w_k^\varepsilon(t) \frac{x-\psi_k(t)}{\left|x-\psi_k(t)\right|^2} \D t.
\end{align}
It is clear that $\nabla S_k^\varepsilon\in L^2(\Omega\backslash U_k)$ so it suffices to prove that $\nabla S_k^\varepsilon\in L^2(\Omega \cap U_k)$.
Besides, we have by the change of variables $x=\Psi_k(y)$
\begin{align}\label{intexleqintey}
\int_{\Omega \cap U_k}  \left|  \nabla S_k^\varepsilon  (x)    \right|^2     \D x &\leq C
		  \int_{\left( \R_-^* \times \R \right) \cap U_0 }  \left| \nabla  S_k^\varepsilon  \Big(\Psi_k(y)\Big)    \right|^2     \D y.
\end{align}
Let us then compute thanks to \eqref{nablaS}
\begin{align}
\nabla  S_k^\varepsilon  \Big(\Psi_k(y)\Big)&=\frac{1}{\|w_k^\varepsilon\|_1}\int_{T^-_{k,\varepsilon}}^{T^+_{k,\varepsilon}} w_k^\varepsilon(t)  \frac{\Psi_k(y)-\Psi_k(0,t)}{\left|\Psi_k(y)-\Psi_k(0,t)\right|^2} \D t\\
		&=\frac{1}{\|w_k^\varepsilon\|_1}\int_{T^-_{k,\varepsilon}}^{T^+_{k,\varepsilon}} w_k^\varepsilon(t) \frac{D_0\Psi_k\Big(y-(0,t)\Big) +O\left( | y-(0,t) |^2 \right) }{\left|D_0\Psi_k\Big(y-(0,t)\Big)\right|^2+O\left( | y-(0,t) |^3 \right)} \D t\\
		&=\frac{D_0\Psi_k}{\|w_k^\varepsilon\|_1}\int_{T^-_{k,\varepsilon}}^{T^+_{k,\varepsilon}} w_k^\varepsilon(t) \frac{\Big(y-(0,t)\Big) +O\left( | y-(0,t) |^2 \right) }{\left|y-(0,t)\right|^2+O\left( | y-(0,t) |^3 \right)} \D t\\
		&=\frac{D_0\Psi_k}{\|w_k^\varepsilon\|_1}\int_{T^-_{k,\varepsilon}}^{T^+_{k,\varepsilon}} w_k^\varepsilon(t) \frac{\Big(y-(0,t)\Big) +O\left( | y-(0,t) |^2 \right) }{\left|y-(0,t)\right|^2} \Big(1+O\big( | y-(0,t) | \big) \Big)  \D t\\
		&=\frac{D_0\Psi_k}{\|w_k^\varepsilon\|_1}\int_{T^-_{k,\varepsilon}}^{T^+_{k,\varepsilon}} w_k^\varepsilon(t) \frac{\Big(y-(0,t)\Big) }{\left|y-(0,t)\right|^2}   \D t +O(1) \label{nablaS1}
\end{align}
where we used the fact that $D_0\Psi_k$ is unitary. The remainder term
is bounded in $\varepsilon$ uniformly with respect to $y\in \left( \R_-^* \times \R \right) \cap U_0$.
Let us estimate each component of the last integral from \eqref{nablaS1}.
We have
\begin{align}\label{nablaS2}
\left|\int_{T^-_{k,\varepsilon}}^{T^+_{k,\varepsilon}} w_k^\varepsilon(t) \frac{y_1 }{y_1^2+(y_2-t)^2}   \D t\right|
&\leq  \|w_k^\varepsilon\|_\infty |y_1| \int_{T^-_{k,\varepsilon}}^{T^+_{k,\varepsilon}}  \frac{1 }{y_1^2+(y_2-t)^2}   \D t\\
		&\leq C_\varepsilon\left[ \arctan\left(\frac{t-y_2}{|y_1|}\right)  \right]_{T^-_{k,\varepsilon}}^{T^+_{k,\varepsilon}} \in L^\infty\left(  \left( \R_-^* \times \R \right) \cap U_0  \right)
\end{align}
and
\begin{align}\label{nablaS3}
&\left|\int_{T^-_{k,\varepsilon}}^{T^+_{k,\varepsilon}} w_k^\varepsilon(t) \frac{y_2-t }{y_1^2+(y_2-t)^2}   \D t\right|\leq \|w_k^\varepsilon\|_\infty  \int_{T^-_{k,\varepsilon}}^{T^+_{k,\varepsilon}}  \frac{|y_2-t| }{y_1^2+(y_2-t)^2}   \D t \\
		&\qquad \leq C_\varepsilon \left( \int_{T^-_{k,\varepsilon}}^{y_2}  \frac{y_2-t }{y_1^2+(y_2-t)^2}   \D t +\int^{T^+_{k,\varepsilon}}_{y_2}  \frac{t-y_2 }{y_1^2+(y_2-t)^2}   \D t  \right) \\
		&\qquad \leq C_\varepsilon \left(  \left[ \ln\Big(y_1^2+(y_2-t)^2\Big)  \right]_{y_2}^{T^+_{k,\varepsilon}} -  \left[ \ln\Big(y_1^2+(y_2-t)^2\Big)  \right]^{y_2}_{T^-_{k,\varepsilon}} \right) \\
		& \qquad  \leq C_\varepsilon \bigg(  \ln\Big(y_1^2+(y_2-T^+_{k,\varepsilon})^2\Big) + \ln\Big(y_1^2+(y_2-T^-_{k,\varepsilon})^2\Big)  -4  \ln |y_1|   \bigg) \in L^2\left(  \left( \R_-^* \times \R \right) \cap U_0  \right).
\end{align}
Putting \eqref{intexleqintey}, \eqref{nablaS1}, \eqref{nablaS2} and \eqref{nablaS3} together allows to conclude the proof.
\end{proof}

\hip
Finally, we are now in position to state the two following results,
whose proofs are postponed to Appendix~\ref{appf0}.

\begin{lem}\label{f0}
For all $k\in \llbracket 1,N \rrbracket$, it holds
$$S_k^\varepsilon=S_k^0+O\left( \varepsilon_k^{1/3} \right)\qquad \text{in } L^2(\Omega)$$
and
$$R_k^\varepsilon=R_k^0+O\left( \varepsilon_k^{1/3} \right)\qquad \text{in } W^{1,3}(\Omega), $$
where $S_k^0(x)=\ln |x-x^{(k)}|$ and $R_k^0$ is the solution of
\begin{equation}\label{defR0}
\left \{
    \begin{aligned}
        \Delta  R_k^0  & = -a_k
        &&\text{ on }
 \Omega, \\
        \partial_{n}  R_k^0 & =- N_k(x,0)  &&\text{ on } 
\partial \Omega\\
	R_k^0(x^{(k)})&=0.
    \end{aligned}
    \right . 
\end{equation}
In particular, the functions $(f_\gamma^\varepsilon)_{|\gamma|>1}$ defined in \eqref{deff} and \eqref{fgamma} satisfy \eqref{limf}.
\end{lem}

As a corollary of the previous result, let us emphasize the following
estimate that will be used below: since
$W^{1,3}(\Omega)$ is continuously imbedded in
$\mathcal C^{1/3}(\overline{\Omega})$, one has:
$$ R_k^\varepsilon(x)=R_k^0(x^{(j)})+O\left( \varepsilon_k^{1/3} + \varepsilon_j^{1/3}
\right)\qquad \text{uniformly on }\Gamma_j^\varepsilon.$$

\begin{cor}\label{corf0}
Let $k,j\in \llbracket 1,N \rrbracket$.
It holds
$$S_k^\varepsilon(x)=-|s_k^\varepsilon| \delta_{k,j}+ S_k^0(x^{(j)})(1-\delta_{k,j})+O\left(| \varepsilon|^{1/6} \right) \qquad \text{uniformly on }\Gamma_j^\varepsilon$$
and
$$ R_k^\varepsilon(x)=R_k^0(x^{(j)})(1-\delta_{k,j})+O\left( |\varepsilon|^{1/3} \right)\qquad \text{uniformly on }\Gamma_j^\varepsilon.$$
In particular, the function $f_k^\varepsilon$ defined in \eqref{deff} satisfies
$$f_k^\varepsilon(x)=- |s_k^\varepsilon| \delta_{k,j} +f_k^0(x^{(j)})(1-\delta_{k,j})+O\left( |\varepsilon|^{1/6} \right) \qquad \text{uniformly on }\Gamma_j^\varepsilon.$$
\end{cor}

\subsection{Determination of the coefficients of $\varphi^\varepsilon$}\label{bord}

We are now in position to determine the coefficients appearing in the
definition~\eqref{defphi} of $\varphi^\varepsilon$, starting
with $(K^\varepsilon_k)_{k \in \llbracket 1,N \rrbracket}$.


\begin{lem}\label{cercle1}
For all $k\in \llbracket 1,N \rrbracket $, 
it holds 
$$\varphi^\varepsilon(x)=1- K_k^\varepsilon |s^\varepsilon_k |+O\Big(|K^\varepsilon|\Big)+O\Big( K_k^\varepsilon |s^\varepsilon_k| |K^\varepsilon| \Big)+O\left( |\varepsilon|^{1/6} \right)$$
uniformly in $x\in \Gamma_k^\varepsilon$.
\end{lem}

\begin{proof}
Let $x\in \Gamma_k^\varepsilon$.
We have thanks to \eqref{defphi} 
and Corollary \ref{corf0}
\begin{align}\label{decoup1}
\varphi^\varepsilon(x)&=1- K_k^\varepsilon |s^\varepsilon_k |+  \sum_{j\neq k}  K^\varepsilon_{j} f_{j}^0(x^{(k)})+O\left( |\varepsilon|^{1/6} \right)+\sum_{m=2}^M \mathop{\sum_{\gamma \in \N^N;}}_{\substack{|\gamma|=m}} (K^\varepsilon)^\gamma f_\gamma^\varepsilon(x) \\
		&=1- K_k^\varepsilon |s^\varepsilon_k |+O\Big(|K^\varepsilon|\Big)+O\left( |\varepsilon|^{1/6} \right)+\sum_{m=2}^M \mathop{\sum_{\gamma \in \N^N;}}_{\substack{|\gamma|=m}} (K^\varepsilon)^\gamma f_\gamma^\varepsilon(x).
\end{align}
Similarly, we have 
$$\sum_{m=2}^M \mathop{\sum_{\gamma \in \N^N;}}_{\substack{|\gamma|=m}} (K^\varepsilon)^\gamma f_\gamma^\varepsilon(x)  =O\Big( K_k^\varepsilon |s^\varepsilon_k| |K^\varepsilon| \Big)+O\left( |K^\varepsilon|^2 \right)$$
thanks to Corollary \ref{corf0}, \eqref{fgamma} and
\eqref{uholder}. 
This concludes the proof.
\end{proof}
\hip
In view of Lemma \ref{cercle1} and in order to cancel the leading term
of $\varphi^\varepsilon|_{\Gamma_{\mathcal D}}$, we set for $k\in \llbracket 1,N \rrbracket$
\begin{align}\label{defK}
K_k^\varepsilon=\frac{1}{|s^\varepsilon_k|}=\left| \ln \left(\frac{\varepsilon_k}{2}\right)\right|^{-1}
\end{align}
according to Lemma \ref{Sbord}.
Notice that the $O(|\varepsilon|^{1/6})$ remainder is then absorbed by
the $O(|K^\varepsilon|)$ (it actually becomes of order $O\left(
  |K^\varepsilon|^\infty \right)$) so that:
$$\varphi^\varepsilon(x)=O\Big(|K^\varepsilon|\Big)$$
uniformly in $x\in \Gamma_k^\varepsilon$.
With this definition of the coefficients $(K^\varepsilon_k)$, it also
holds (see Corollary~\ref{corf0})
\begin{align}\label{bornesf}
f_k^\varepsilon(x)=- (K_k^\varepsilon)^{-1} +O\left( |\varepsilon|^{1/6} \right)
\end{align}
uniformly in $x\in \Gamma_k^\varepsilon$. 

Let us now provide a more precise expansion of the function
$\varphi^\varepsilon$ which will be useful to determine the parameters $(b_{\beta,j})\mathop{}_{\substack{1\leq |\beta|\leq M \\  j\in
    \mathrm{supp}\, \beta}}$.
\begin{lem}\label{coeffcercle}
Let $k\in \llbracket 1,N \rrbracket$. 
The function $\varphi^\varepsilon$ satisfies 
\begin{align}
\varphi^\varepsilon(x)&= 
\mathop{
\sum_{\gamma \in \N^N;}}_{\substack{0\leq | \gamma|\leq M-1}}
(K^\varepsilon)^\gamma \tilde \varphi^{(k)}_\gamma+\mathop{
\sum_{\gamma \in \N^N;}}_{\substack{ | \gamma|= M}}(K^\varepsilon)^\gamma\left(
v_\gamma(x^{(k)})+
 \mathop{\sum_{j\in \mathrm{supp}\, \gamma;}}_{\substack{j\neq k}} b_{\gamma,j}f_{j}^0(x^{(k)}) 
\right) +O\left( |K^\varepsilon|^\infty \right)
\end{align}
uniformly in $x\in \Gamma_k^\varepsilon$,
where $\tilde \varphi_0^{(k)}=0$ and 
for $|\gamma|\geq 1$,
$$\tilde \varphi^{(k)}_\gamma=v_\gamma(x^{(k)})+
 \mathop{\sum_{j\in \mathrm{supp}\, \gamma;}}_{\substack{j\neq k}} b_{\gamma,j}f_{j}^0(x^{(k)}) - b_{\gamma+e_k,k} 
. $$
In particular, it holds
$$\varphi^\varepsilon(x)=\mathop{
\sum_{\gamma \in \N^N;}}_{\substack{0\leq | \gamma|\leq M-1}}
(K^\varepsilon)^\gamma \tilde \varphi^{(k)}_\gamma+O\left(|K^\varepsilon|^M\right)$$
uniformly in $x\in \Gamma_k^\varepsilon$.
\end{lem}

\begin{proof}
Let $x\in \Gamma_k^\varepsilon$, let us indicate explicitly the
dependency of $\varphi^\varepsilon$ on $M$ using the notation $\varphi_M^\varepsilon=\varphi^\varepsilon$ and let us proceed by induction on $M\in \N_{\geq 1}$.

When $M=1$, the statement writes thanks to 
\eqref{c1} as
\begin{align}\label{amq}
\varphi_1^\varepsilon (x)=
\sum_{j\neq k} K_j^\varepsilon f_j^0(x^{(k)})
+O\left(|K^\varepsilon|^{\infty}\right)
\end{align}
which holds true as a consequence of \eqref{decoup1} and \eqref{defK}.

Suppose now that $M> 1$ and that the result holds true for the function $\varphi^\varepsilon_{M-1}.$
We have 
\begin{align}\label{phiphitilde}
\varphi_M^\varepsilon=  \varphi_{M-1}^\varepsilon + \mathop{
\sum_{\gamma \in \N^N;}}_{\substack{ | \gamma|=M}} (K^\varepsilon)^\gamma  f_\gamma^\varepsilon.
\end{align}
Besides, combining Corollary 
\ref{corf0} 
and \eqref{uholder} yields
\begin{align}
f_\gamma^\varepsilon(x)&=v_\gamma(x)+
\sum_{j\in \mathrm{supp}\, \gamma;}
 b_{\gamma,j}f_{j}^\varepsilon(x) 
 \\
		&=v_\gamma(x^{(k)})+
 \mathop{\sum_{j\in \mathrm{supp}\, \gamma;}}_{\substack{j\neq k}} b_{\gamma,j}f_{j}^0(x^{(k)}) -b_{\gamma,k}(K_k^\varepsilon)^{-1} \1_{\gamma_k\neq 0} 
+O\Big( |K^\varepsilon|^\infty \Big). \label{expfgamma}
\end{align}
Note that the third term 
depends on $\varepsilon$.
About this particular term, we have
$$\mathop{
\sum_{\gamma \in \N^N;}}_{\substack{ | \gamma|=M}} (K^\varepsilon)^\gamma b_{\gamma,k}(K_k^\varepsilon)^{-1} \1_{\gamma_k\neq 0} = \mathop{
\sum_{\gamma \in \N^N;}}_{\substack{ | \gamma|=M-1}} (K^\varepsilon)^\gamma b_{\gamma+e_k,k}$$
which combined with \eqref{phiphitilde} and \eqref{expfgamma} shows that the result also holds for $\varphi^\varepsilon_M$ and the proof is therefore complete.
\end{proof}

We are now in position to determine the parameters
$(b_{\beta,j})\mathop{}_{\substack{1\leq |\beta|\leq M \\  j\in
    \mathrm{supp}\, \beta}}$ so that the quasimode
$\varphi^\varepsilon$ satisfies the homogeneous Dirichlet boundary
conditions on $\Gamma^\varepsilon_{\mathcal D}$ up to an
error of order $O(|K^\varepsilon|^M)$.

\begin{lem}\label{resolcercle}
There exists a unique choice of the family
$$
(b_{\beta,j})\mathop{}_{\substack{1\leq |\beta|\leq M \\  j\in \mathrm{supp}\, \beta}}    $$
such that for all $k\in \llbracket 1,N \rrbracket$ and $\gamma\in \N^N$ satisfying $0\leq |\gamma|\leq M-1$, 
we have 
$
\tilde \varphi^{(k)}_\gamma=0.$
This family is defined by induction as follows: For $|\beta|=1$,
$b_{\beta,j}=1$ and for $|\beta|>1$,
\begin{align}\label{formb}
b_{\beta,j}=v_{\beta-e_j}(x^{(j)})
+\mathop{\sum_{n\in \mathrm{supp}\, (\beta-e_j);}}_{\substack{n\neq j}} b_{\beta-e_j,n} f_{n}^0(x^{(j)}).
\end{align}
\end{lem}

\begin{proof}
We once again proceed by induction on $M\in \N_{\geq 1}$.
When $M=1$, the result follows immediately from 
 \eqref{c1} 
and the fact that $\tilde \varphi_0^{(k)}=0$ (see Lemma \ref{coeffcercle}).

Suppose now that $M> 1$ and that there exist
$$
 (b'_{\beta,j})\mathop{}_{\substack{1\leq |\beta|\leq M-1 \\  j\in \mathrm{supp}\, \beta}}    $$
such that for all $k\in \llbracket 1,N \rrbracket$ and $\gamma\in \N^N$ satisfying $0\leq |\gamma|\leq M-2$, 
we have 
$
\tilde \varphi^{(k)}_\gamma=0.$
Since the $(\tilde \varphi^{(k)}_\gamma)_{0\leq |\gamma|\leq M-2}$ only depend on $
(b_{\beta,j})_{1\leq |\beta|\leq M-1}$, we must take (by the
uniqueness of the induction hypothesis) for $1\leq |\beta| \leq M-1$
$$
(b_{\beta,j})_{j\in \mathrm{supp}\, \beta}=(b'_{\beta,j})_{j\in \mathrm{supp}\, \beta}$$
and it only remains to find 
$(b_{\beta,j})_{ |\beta|= M} $ 
such that for all $k\in \llbracket 1,N \rrbracket$ and $\gamma\in \N^N$ satisfying $|\gamma|= M-1$, 
we have 
$
\tilde \varphi^{(k)}_\gamma=0.$
To this aim, we denote
$$E=\left\{(\beta,j)\in \N^N \times \llbracket 1,N \rrbracket\, ; \, |\beta|=M \text{ and } j\in \mathrm{supp} \, \beta\right\} \quad \text{and} \quad F=\{(\gamma,j) \in \N^N \times \llbracket 1,N \rrbracket\, ; \, |\gamma|=M-1\}.$$ 
Since the map
$$
\left\{
  \begin{aligned}
E &\to F\\
(\beta,j)&\mapsto (\beta-e_j,j)
\end{aligned}
\right.
$$
is a bijection, it induces an isomorphism between $\R^E$ and $\R^F$.
Besides, it appears from the definition of $\tilde \varphi^{(k)}_\gamma$ that the map
$$\left\{
  \begin{aligned}
\R^E&\to \R^E\\
(b_{\beta,j})&\mapsto \left(\tilde \varphi^{(j)}_{\beta-e_j}\right)
\end{aligned}
\right.
$$
is affine and that its linear part is $ -\mathrm{Id}$.
This proves that there exists a unique family 
$$(b_{\beta,j})\mathop{}_{\substack{|\beta|= M \\  j\in \mathrm{supp}\, \beta}}    $$
such that for all $k\in \llbracket 1,N \rrbracket$ and $\gamma\in \N^N$ satisfying $|\gamma|= M-1$, 
we have 
$
\tilde \varphi^{(k)}_\gamma=0.$
Moreover, its expression is given by \eqref{formb} thanks to the expression of $\tilde \varphi^{(k)}_\gamma$.
This concludes the proof.
\end{proof}

In the following, for all $1\leq|\beta|\leq M$, the values of
$(b_{\beta,j})_{j\in \mathrm{supp}\, \beta}$ are those given by Lemma \ref{resolcercle}.
Recall that thanks to Lemma \ref{lemcomp} and \eqref{compgamma}, it also fixes the value of $a_\beta$.
Now that this choice has been made, let us conclude this section by
summarizing the properties satisfied by the quasimode $\varphi^\varepsilon$.

\begin{prop}\label{propphi}
The function $\varphi^\varepsilon$ is such that :
\begin{enumerate}[label=\alph*)]
\item It satisfies $\varphi^\varepsilon =  1 + O(|K^\varepsilon|)$ in
  $L^2(\Omega)$. In particular, $\|\varphi^\varepsilon\|_2=|\Omega|^{1/2}+O(|K^\varepsilon|)$.  \label{phiO1}
\item $\varphi^\varepsilon \in H^1(\Omega)$.     \label{phiH1}
\item It satisfies $\partial_n \varphi^\varepsilon=0$ on $\Gamma_{\mathcal N}^\varepsilon$. \label{phiNeum}
\item It satisfies $\varphi^\varepsilon(x)=O\left(|K^\varepsilon|^M\right)$ uniformly in $x\in \Gamma^\varepsilon_{\mathcal D}$.  \label{phiBord}
\item It satisfies the approximate eigenvalue problem
\begin{align}\label{vpappro2}
\bigg( -\Delta -
 \mathop{\sum_{\gamma \in \N^N;}}_{\substack{1\leq |\gamma|\leq M}} a_\gamma (K^\varepsilon)^\gamma \bigg) \varphi^\varepsilon=O\left( |K_\varepsilon|^{M+1} \right) \qquad \text{in } L^2(\Omega).
\end{align}  \label{phiDelta}
\end{enumerate}
%
%
\end{prop}

\begin{proof}
Item \ref{phiO1} is a direct consequence of Lemma \ref{f0} and \eqref{defphi}.
Item \ref{phiH1} follows from \eqref{uholder}, \eqref{RW1} and Lemma \ref{SH1}.
Item \ref{phiNeum} is obtained thanks to Corollary \ref{Rhold} and \eqref{vgam}.
Item \ref{phiBord} is a consequence of Lemmas \ref{coeffcercle} and \ref{resolcercle}.
Finally, item \ref{phiDelta} follows from Proposition \ref{blibre}, Corollary \ref{Rhold}, Lemma \ref{f0} and \eqref{defK}.
\end{proof}

\section{Mean exit time}\label{sec:MET}

Let us first give a first estimate on $\lambda_0^\varepsilon$.

\begin{thm}\label{gap}
There exists $c>0$ such that for all $\varepsilon \in (\mathbb{R}_+^*)^{N}$ with $|\varepsilon|$ small enough, we have
$$\mathrm{Spec}(
\mathcal L
^{\varepsilon})\cap [0,c] = \{
\lambda_0^\varepsilon\}.$$
Moreover, it holds $
\lambda_0^\varepsilon=O\left(|K^\varepsilon|\right)$.
\end{thm}
\begin{proof}
As mentioned in Section~\ref{sec:qoi}, the operator $\mathcal
L^\varepsilon$  has discrete spectrum. By using the min-max principle, its
second eigenvalue is bounded from below by the second eigenvalue of
the opposite of the Laplace operator with full Neumann boundary conditions on
$\partial \Omega$: this proves the first statement of the theorem.

Concerning the estimate of $\lambda_0^\varepsilon$, the proof is very close to the one of~\cite[Theorem 1.2]{LeRaSt},
where one considers the case when $\Omega$ is a disk. Let us provide
the proof for completeness.
Recall that
\begin{equation}
    \label{eq:rayleigh}
    \lambda_0^\varepsilon 
    = \inf_{u \in H^1(\Omega), u=0 \text{ on } 
    \Gamma_{\mathcal D}^\varepsilon} \frac{\int_{\Omega} |\nabla u|^2}{\int_{\Omega} |u|^2},
\end{equation}
so that  $\lambda_0^\varepsilon \ge 0$. In order to prove the upper
bound, we build a simple quasimode as follows.


Consider the modified domain $\Omega_\varepsilon = \Omega \backslash
\cup_{k=1}^N B(x^{(k)}, \varepsilon_k)$
whose boundary is again split into two subsets:
${\overline{\Gamma}_{\mathcal D}^\varepsilon = \cup_{k=1}^N \partial B(x^{(k)},
  \varepsilon_k)} \cap \partial \Omega_\varepsilon$ and the complementary set
$\Gamma_{\mathcal N}^\varepsilon$.
Notice that $\lambda_0^\varepsilon \le \bar \lambda_0^\varepsilon$
where
\begin{equation}
    \label{eq:rayleigh_tilde}
  \bar   \lambda_0^\varepsilon 
    = \inf_{\bar u \in H^1(\Omega_\varepsilon), \bar u=0 \text{ on } 
   \bar \Gamma_{\mathcal D}^\varepsilon}
 \frac{\int_{\Omega_\varepsilon} |\nabla \bar
   u|^2}{\int_{\Omega_\varepsilon} |\bar u|^2},
\end{equation}
since any function $\bar u \in H^1(\Omega_\varepsilon)$ such that $
\bar u=0$ on 
   $\bar \Gamma_{\mathcal D}^\varepsilon$ can be extended to a
 function $ u \in H^1(\Omega), u=0$ on  
   $ \Gamma_{\mathcal D}^\varepsilon$ by considering $u=\bar u
  1_{\Omega_\varepsilon}$, and that then,
$\frac{\int_{\Omega} |\nabla (\bar
   u 1_{\Omega_\varepsilon})|^2}{\int_{\Omega} |\bar u 1_{\Omega_\varepsilon}|^2}=  \frac{\int_{\Omega_\varepsilon} |\nabla \bar
   u|^2}{\int_{\Omega_\varepsilon} |\bar u|^2}$.

 Let us now
 introduce the   quasimode $\varphi^\varepsilon_\alpha$ defined as,
 for some $\alpha > 0$ to be defined later on, 
\[
    \varphi^\varepsilon_\alpha = 1 + \frac{1}{\alpha}\sum_{k=1}^{N} K^\varepsilon_k f_k(x),
\]
with $f_k(x) = \ln(|x - x^{(k)}|) + R_k(x)$, with $R_k(x) =\tilde  R_k(x,0)$, see Equation~\eqref{eq:Rkepd} and
Proposition~\ref{solR} for the
definition of $\tilde R_k(x,t)$. Let us emphasize that $R_k$ and $f_k$ do not depend on
$\varepsilon$. From the previous results, we know that $f_k
\in H^1 (\Omega_\varepsilon)$, and $\partial_n f_k = 0$ on
$\Gamma^\varepsilon_{\mathcal N}$ for any $\varepsilon > 0$. Moreover, from
Proposition~\ref{solR}, $\|R_k\|_{\infty} < \infty$ and $\|R_k\|_{H^1(\Omega)} < \infty$.

Now, for a fixed $k\in \llbracket 1,N \rrbracket$, 
we have that, for $x \in \partial B(x^{(k)},\varepsilon_k)$,
\begin{align*}
    \varphi_\alpha^\varepsilon(x) 
    & = 1 + \frac{1}{\alpha} K^\varepsilon_k \left(  \ln|x - x^{(k)}| + \mathrm{O}\left( 1\right)
    \right) + \mathrm{O}(|K^\varepsilon|)\\
    & = 1 - \frac{1}{\alpha} \Big( 1  + \mathrm{O}\left(K^\varepsilon_{k}\right)
    \Big) + \mathrm{O}(|K^\varepsilon|)
\end{align*}
 in
$L^\infty(\Omega_\varepsilon)$.
Hence, for any $\varepsilon$ small enough, there exists $\alpha^* = \alpha^*(\varepsilon)$ such that 
\[
    \varphi_{\alpha^*}^\varepsilon \leq 0 \text{ on } \overline
    \Gamma_{\mathcal D}^\varepsilon.
\]
Therefore, 
$\left. \varphi_{\alpha^*}^\varepsilon \right|_+ =
\max(\varphi_{\alpha^*}^\varepsilon,0)$ is a function in
$H^1(\Omega_\varepsilon)$ such that,
$\left. \varphi_{\alpha^*}^\varepsilon \right|_+=0$ on
$\overline{\Gamma}_{\mathcal D}^\varepsilon $. 
Notice that one can choose $\alpha^*$ such that $\alpha^*$ goes to
$1$ in the limit $\varepsilon \to 0$, so that we assume in the
following that $\alpha^*$ is bounded from below and from
above by a positive constant
independent of $\varepsilon$.
Let us now evaluate the Rayleigh quotient introduced in~\eqref{eq:rayleigh_tilde} for $\bar u = \left. \varphi_{\alpha^*}^\varepsilon \right|_+$. 
First the denominator $
\vert|\left. \varphi_{\alpha^*}^\varepsilon\right|_+
\vert|_{L^2(\Omega_\varepsilon)}^2$ is bounded from below by a
positive constant
independent of $\varepsilon$ as 
\begin{equation}
    \label{eq:bound denominator}
    \left. \varphi_{\alpha^*}^\varepsilon \right|_+ =  1 + \mathrm{O}_{L^2(\Omega_\varepsilon)}(|K^\varepsilon|).
\end{equation}
For the numerator, one gets,
$$\int_{\Omega_\varepsilon} |\nabla
\left. \varphi_{\alpha^*}^\varepsilon\right|_+|^2 \le
C \sum_{k=1}^N  (K^\varepsilon_k)^2 \left(  \int_{\Omega_\varepsilon} |\nabla \ln |x -
  x^{(k)}| |^2\, dx + \|R_k\|_{H^1(\Omega)}^2 \right)  $$ and 
using polar coordinates, for some $\rho > 0$ (independent of
$\varepsilon$, since $\Omega$ is bounded),
$$
    (K^\varepsilon_k)^2 \int_{\Omega_\varepsilon} |\nabla \ln |x -
  x^{(k)}||^2\, dx
  \leq C  (K^\varepsilon_k)^2 \int_{\varepsilon_k}^\rho
   \frac{1}{r^{2}} r \, dr \leq (K^\varepsilon_k)^2
             C |\ln(\varepsilon_k)| \leq C K^\varepsilon_k,
$$
for positive constant $C$ independent of $\varepsilon$, whose value change from one occurrence to another. Thus, 
one has that
\begin{equation} \label{eq:bound gradient}
    \int_{\Omega_\varepsilon} |\nabla \varphi_{\alpha^*}^\varepsilon|^2
    \leq C |K^\varepsilon|.
\end{equation}
Finally, 
plugging \eqref{eq:bound denominator} and \eqref{eq:bound gradient} in \eqref{eq:rayleigh_tilde} leads to
\[
\lambda_0^\varepsilon \le    \bar \lambda_0^\varepsilon \leq C |K^\varepsilon|, 
\] 
which concludes the proof. 
  \end{proof}

From now on, we denote $
 u_0^\varepsilon$ an eigenfunction of $
\mathcal L
^\varepsilon$ associated to $
 \lambda_0^\varepsilon$.
It is known (see for instance \cite[Theorem 8.38]{Gilbarg-Trudinger})
that $\lambda_0^\varepsilon$ is a simple eigenvalue and that $u_0^\varepsilon$ has a sign on $\Omega_\varepsilon$.
We can thus choose to renormalize it so that
\begin{align}\label{u,1}
\langle u_0^\varepsilon, \1 \rangle=-\|u_0^\varepsilon\|_{1}=-1.
\end{align}
where here and in the following, $\langle \cdot, \cdot \rangle$
denotes the $L^2$ scalar product on $L^2(\Omega)$.
%
%
%
%
This will in particular be useful to establish the following lemma.

\begin{lem}\label{phi,u}
It holds 
$$\|u_0^\varepsilon\|_2=|\Omega|^{-1/2}+O(|K^\varepsilon|)$$
and
$$\langle \varphi^\varepsilon , u_0^\varepsilon \rangle = -1+O(|K^\varepsilon|).$$
\end{lem}

\begin{proof}
Let us start with the first statement.
Applying the Poincaré-Wirtinger inequality to $u^\varepsilon_0$ (since $\Omega$ is bounded) gives
\begin{align}\label{pw}
\left\|u_0^\varepsilon-\frac{1}{|\Omega|}\langle u_0^\varepsilon  ,
  \1\rangle\right\|_2^2\leq C \|\nabla u_0^\varepsilon\|_2^2= C
  \lambda_0^\varepsilon \|u_0^\varepsilon\|_2^2\leq C |K^\varepsilon| \|u_0^\varepsilon\|_2^2
\end{align}
thanks to Theorem \ref{gap}.
By \eqref{u,1}, the left-hand side satisfies
\begin{align}\label{lhs}
\left\|u_0^\varepsilon-\frac{1}{|\Omega|}\langle u_0^\varepsilon  ,
  \1\rangle\right\|_2^2=\|u_0^\varepsilon\|_2^2-\frac{1}{|\Omega| }.
\end{align}
Combining \eqref{pw}, \eqref{lhs}, we obtain 
\begin{align}\label{mino}
\|u_0^\varepsilon\|_2=  |\Omega|^{-1/2}+O\left( |K^\varepsilon| \right)
\end{align}
which is the first statement. 
Since $\varphi^\varepsilon=1+O(|K^\varepsilon|)$ in $L^2(\Omega)$,
the second statement is a consequence of the first one and of \eqref{u,1}.
The proof is therefore complete.
\end{proof}
\hip

We are now in position to state the first main result.
\begin{thm}\label{lam0}
We have 
$$ \lambda_0^\varepsilon=\tilde \lambda_0^\varepsilon+O\Big( |K^\varepsilon|^{M+1} \Big) $$
with
\begin{align}\label{lamtilde}
\tilde\lambda_0^\varepsilon=\mathop{\sum_{\gamma \in \N^N;}}_{\substack{1\leq |\gamma|\leq M}} a_\gamma (K^\varepsilon)^\gamma
\end{align}
where we recall that the coefficients $a_\gamma$ are fixed by Lemma \ref{resolcercle} combined with Lemma \ref{lemcomp} and \eqref{compgamma}.

\end{thm}

\begin{proof}
Let us use a Green formula (which is licit
thanks to item \ref{phiH1} from Proposition \ref{propphi} and since $\Delta \varphi^\varepsilon \in L^2(\Omega)$) to write
\begin{align}\label{lam01}
\lambda_0^\varepsilon=\frac{\lambda_0^\varepsilon \langle \varphi^\varepsilon, u_0^\varepsilon \rangle}{ \langle \varphi^\varepsilon, u_0^\varepsilon \rangle}=\frac{ \langle \varphi^\varepsilon, -\Delta u_0^\varepsilon \rangle}{ \langle \varphi^\varepsilon, u_0^\varepsilon \rangle}=\frac{ \langle -\Delta \varphi^\varepsilon, u_0^\varepsilon \rangle - \langle  \varphi^\varepsilon, \partial_n u_0^\varepsilon \rangle_{\partial \Omega} + \langle \partial_n \varphi^\varepsilon,  u_0^\varepsilon \rangle_{\partial \Omega} }{ \langle \varphi^\varepsilon, u_0^\varepsilon \rangle}.
\end{align}
The last term is actually zero since $\partial_n \varphi^\varepsilon$ vanishes on $ \Gamma^\varepsilon_{\mathcal N}$ (by Proposition \ref{propphi}, item \ref{phiNeum}), while $u_0^\varepsilon$ vanishes on $\Gamma^\varepsilon_{\mathcal D}$.
We thus have
\begin{align}\label{lam02}
\lambda_0^\varepsilon=\frac{ \langle -\Delta \varphi^\varepsilon,
  u_0^\varepsilon \rangle }{ \langle \varphi^\varepsilon,
  u_0^\varepsilon \rangle} +  \langle  \varphi^\varepsilon,
  \partial_n u_0^\varepsilon \rangle_{\partial \Omega}  \left( 1 + O\Big(  |K^\varepsilon| \Big)\right)
\end{align}
thanks to Lemma \ref{phi,u}.
By Proposition \ref{propphi}, item \ref{phiDelta} and still using Lemma \ref{phi,u}, the first term satisfies
$$\frac{ \langle -\Delta \varphi^\varepsilon, u_0^\varepsilon \rangle }{ \langle \varphi^\varepsilon, u_0^\varepsilon \rangle}=\tilde\lambda_0^\varepsilon 
+O\Big( |K^\varepsilon|^{M+1} \Big).$$
In order to control the second one, let us use Proposition~\ref{prop:borne_dnu0} 
as well as item \ref{phiBord} from Proposition \ref{propphi} to write
\begin{align}\label{lambdaesp}
|\langle   \partial_n u_0^\varepsilon , \varphi^\varepsilon
  \rangle_{\partial \Omega}| \le \lambda_0^\varepsilon \, \| \varphi^\varepsilon\|_{L^\infty(\Gamma^\varepsilon_{\mathcal D})}=  O\Big( \lambda_0^\varepsilon   |K^\varepsilon|^{M} \Big)=O\Big(  |K^\varepsilon|^{M+1} \Big) 
\end{align}
where we also used 
%
%
%
Theorem \ref{gap}.
This completes the proof.
\end{proof}
\begin{rema}
  At the first order, Theorem~\ref{lam0} gives
  \begin{align}\label{lam0first}
    \lambda_0^\varepsilon = \frac{\pi}{|\Omega|} \overline{K}^\varepsilon + O(|K^\varepsilon|^2)
  \end{align}
  where we introduced the notation $\overline{K}^\varepsilon = \sum_{k=1}^N K_k^\varepsilon$.
  This is in agreement with~\cite[Theorem 3.1]{Ammari2, PWK, SSH2}, and
  the result of~\cite[Theorem 1.5]{LeRaSt} in the case of the disk, 
  since $|\Omega|=\pi$ in this setting.
\end{rema}

\section{Law of the first exit point}\label{sec:LFEP}

For $k\in \llbracket 1, N \rrbracket$, the aim of this section is to give a precise computation of
$$\frac{\langle \partial_n u_0^\varepsilon, \1_{\Gamma_k^\varepsilon}
  \rangle_{\partial \Omega}}{\langle \partial_n u_0^\varepsilon, \1
  \rangle_{\partial \Omega}}.$$
Notice that the numerator is well defined since $\partial_n
u_0^\varepsilon$ can be identified as a positive Radon measure with
finite mass (see Appendix~\ref{sec:prob_res}).
Of course, this is relevant only when $N\geq 2$, so we now suppose that this is the case.
The denominator can easily be computed thanks to~\eqref{eq:u0}-\eqref{u,1}:
\begin{align}\label{dnuGamma}
\langle \partial_n u_0^\varepsilon, \1 \rangle_{\partial \Omega}=\langle \Delta u_0^\varepsilon, \1 \rangle=\lambda_0^\varepsilon.
\end{align}
We will thus denote for shortness
\begin{align}\label{defXk}
Y^\varepsilon_k=\frac{\langle \partial_n u_0^\varepsilon, \1_{\Gamma_k^\varepsilon} \rangle_{\partial \Omega}}{\lambda_0^\varepsilon}.
\end{align}

\hop
Before we begin our analysis, let us first introduce a few notations.
For a fixed $k_0\in \llbracket 1,N \rrbracket$, we now denote (compare
with~\eqref{defphi}):  for all $x \in \Omega$, 
\begin{align}\label{hatphi}
 \hat\varphi_{k_0}^\varepsilon (x) = 1+ 
 \mathop{\sum_{\gamma \in \N^{N}; \gamma_{k_0}=0}}_{\substack{1\leq |\gamma|\leq M}} (K^\varepsilon)^\gamma \hat f_\gamma^{\varepsilon,(k_0)}(x)
\end{align}
the quasimode associated to the problem where the exit region around $x^{(k_0)}$ has been removed, with
$$\hat f_\gamma^{\varepsilon,(k_0)}=\hat v_\gamma^{(k_0)}+\sum_{k\in \mathrm{supp}\, \gamma}  \hat b^{(k_0)}_{\gamma,k} f_k^\varepsilon.$$
The quantities $\hat v_\gamma^{(k_0)}$ and $\hat
b_{\gamma,k}^{(k_0)}$ are those defined previously (see \eqref{vgam} and Lemma
\ref{resolcercle}), but considered only for some $\gamma$'s such that
$\gamma_{k_0}=0$.
Similarly, we put in view of Theorem \ref{lam0} and \eqref{lamtilde}
\begin{align}\label{hatlam}
\hat \lambda_{0,k_0}^\varepsilon=\mathop{\sum_{\gamma \in \N^N;\gamma_{k_0}=0}}_{\substack{1\leq |\gamma|\leq M}} \hat a_\gamma^{(k_0)} (K^\varepsilon)^\gamma
\end{align}
where, we recall, the $(\hat a_\gamma^{(k_0)})_{\gamma \in
  \N^N;\gamma_{k_0}=0}$ satisfy (see~\eqref{compgamma},~\eqref{eq:F}, Lemma~\ref{lemcomp})
$$a_k= \frac{\pi}{|\Omega|} \text{ for }k\neq k_0 \qquad \text{and} \qquad \hat a_\gamma^{(k_0)}=- \mathop{\sum_{\beta \in \N^N;\; \beta\leq \gamma;}}_{\substack{1\leq |\beta|\leq |\gamma|-1}} \hat a^{(k_0)}_{\gamma-\beta} \langle \hat f_\beta^{0,(k_0)}\rangle +\frac{\pi}{|\Omega|}\sum_{k\in \mathrm{supp}\, \gamma}  \hat b^{(k_0)}_{\gamma,k}$$
for $|\gamma|>1$ such that $\gamma_{k_0}=0$.

The computation of the $(Y^\varepsilon_k)_{k\in \llbracket 1,N \rrbracket}$ will require a precise computation of $(\langle \hat f_\gamma^{\varepsilon,(k_0)}, u_0^\varepsilon \rangle)_{k_0\in \llbracket 1,N \rrbracket}$ which is the point of the next Lemma.

Let us emphasize that the notation
  $(k_0)$ in $\hat f_\gamma^{\varepsilon,(k_0)}$, $\hat
  v_\gamma^{(k_0)}$, $\hat
  a^{(k_0)}_{\gamma}$, and $\hat b^{(k_0)}_{\gamma,k}$ is simply used
  to recall that these coefficients are the ones built previously, but
  considered for a multi-index $\gamma$ such that $\gamma_{k_0}=0$.

\begin{lem}\label{cdl}
Let $k_0\in \llbracket 1, N \rrbracket$ as well as $\gamma\in \N^{N}$ such that $\gamma_{k_0}=0$ and $|\gamma|\geq 1$.
For all $n\in \llbracket 0, |\gamma|-1\rrbracket$, there exist some real numbers $c_n^{(k_0)}(\gamma)$, $(d^{(k_0)}_{n,k}(\gamma))_{k\in \mathrm{supp}\, \gamma}$ and $(\ell^{(k_0)}_{n,k}(\gamma))_{k\in \llbracket 1,N \rrbracket}$ such that
$$(\lambda_0^\varepsilon)^{|\gamma|}\langle \hat f^{\varepsilon,(k_0)}_\gamma, u_0^\varepsilon \rangle=\sum_{n=0}^{|\gamma|-1}(\lambda_0^\varepsilon)^{n} \left( c^{(k_0)}_n(\gamma)+\sum_{k\in \mathrm{supp}\, \gamma} d^{(k_0)}_{n,k}(\gamma) \frac{\lambda_0^\varepsilon}{K_k^\varepsilon} Y^\varepsilon_k + \sum_{k=1}^N   \ell^{(k_0)}_{n,k}(\gamma) \lambda_0^\varepsilon Y^\varepsilon_k  \right) +O\left( |K^\varepsilon|^\infty \right) .$$
These are defined by induction on 
$ |\gamma|\geq 1$ as follows : when $|\gamma|=1$, say $\gamma=e_j$ (with $j\neq k_0$), we have 
\begin{align}\label{cdl0}
c_0^{(k_0)}(e_j)=- \frac{\pi}{|\Omega|}, \qquad d^{(k_0)}_{0,j}(e_j)=1, \qquad \ell^{(k_0)}_{0,k}(e_j)=-f_j^0(x^{(k)})
\1_{k\neq j}
\end{align}
and when $|\gamma|>1$, we have for $n\in \llbracket 0, |\gamma|-2\rrbracket$
\begin{align}\label{cn}
c_n^{(k_0)}(\gamma)=
\mathop{\sum_{\beta \in \N^{N};\,\beta\leq \gamma 
;}}_{\substack{|\gamma|-1-n \leq |\beta|\leq |\gamma|-1}} \hat a^{(k_0)}_{\gamma-\beta}\, c^{(k_0)}_{n-(|\gamma|-1-|\beta|)}(\beta),
\end{align}
\begin{align}\label{dn}
  \text{for } k \in \mathrm{supp}\, \gamma, \quad
d^{(k_0)}_{n,k}(\gamma)=
\mathop{\sum_{\beta \in \N^{N};\,\beta\leq \gamma; \, k\in \mathrm{supp}\, \beta}}_{\substack{|\gamma|-1-n \leq |\beta|\leq |\gamma|-1}} \hat a^{(k_0)}_{\gamma-\beta}\, d^{(k_0)}_{n-(|\gamma|-1-|\beta|),k}(\beta),
\end{align}

\begin{align}\label{ln}
\ell^{(k_0)}_{n,k}(\gamma)=
\mathop{\sum_{\beta \in \N^{N};\,\beta\leq \gamma 
;}}_{\substack{|\gamma|-1-n \leq |\beta|\leq |\gamma|-1}} \hat a^{(k_0)}_{\gamma-\beta}\, \ell^{(k_0)}_{n-(|\gamma|-1-|\beta|),k}(\beta),
\end{align}
while
\begin{align}\label{gamma-1}
c^{(k_0)}_{|\gamma|-1}(\gamma)=- \hat a^{(k_0)}_{\gamma} , \quad
  d^{(k_0)}_{|\gamma|-1,k}(\gamma)=\hat b^{(k_0)}_{\gamma,k}, \quad
  \ell^{(k_0)}_{|\gamma|-1,k}(\gamma)=-\hat v^{(k_0)}_\gamma(x^{(k)})-
  \!\!\!\sum_{j\in \mathrm{supp}\, \gamma \backslash \{k\}} \!\!\!\hat b^{(k_0)}_{\gamma,j} f_j^0(x^{(k)}).
\end{align}

\end{lem}

\begin{proof}
Let us proceed by induction on $|\gamma|\geq 1$ for $\gamma$ satisfying $\gamma_{k_0}=0$.

When $|\gamma|=1$, say $\gamma=e_k$ (with $k\neq k_0$), one has (thanks to \eqref{c1}) $\hat
f^{\varepsilon,(k_0)}_\gamma=f^\varepsilon_k$, and proving the Lemma
amounts to proving that
\begin{align}\label{amq3}
\lambda_0^\varepsilon\langle f_k^\varepsilon , u_0^\varepsilon \rangle = -\frac{\pi}{|\Omega|}+\frac{\lambda_0^\varepsilon}{K^\varepsilon_k}Y^\varepsilon_k - \lambda_0^\varepsilon\sum_{j\in \llbracket 1,N \rrbracket \backslash \{k\}} f_k^0(x^{(j)}) Y_j^\varepsilon+O\left( |K^\varepsilon|^\infty \right).
\end{align}
Let us do the computation : using Lemma \ref{SH1} and Corollary \ref{Rhold}, we have
\begin{align}
\lambda_0^\varepsilon \langle f_k^\varepsilon, u_0^\varepsilon \rangle&=
 \langle f_k^\varepsilon, -\Delta u_0^\varepsilon \rangle\\
		&=
 \langle -\Delta f_k^\varepsilon,  u_0^\varepsilon \rangle -  \langle f_k^\varepsilon, \partial_n u_0^\varepsilon \rangle_{\Gamma^\varepsilon_{\mathcal D}}  
 \qquad \text{since }u_0^\varepsilon \text{ (resp. }\partial_n f_k^\varepsilon \text{) vanishes on }\Gamma_{\mathcal D}^\varepsilon \text{ (resp. on }\partial \Omega\backslash\Gamma_{\mathcal D}^\varepsilon\text{)}\\
		&=
 \frac{\pi}{|\Omega|} \langle \1 ,  u_0^\varepsilon \rangle - \sum_{j=1}^N \langle f_k^\varepsilon, \partial_n u_0^\varepsilon \rangle_{\Gamma^\varepsilon_j}  
 \qquad \text{by \eqref{eq:fkepd}, Lemma~\ref{lemcomp}}\\
		&=- \frac{\pi}{|\Omega|}   + 
 \frac{1}{K_k^\varepsilon}\langle \1, \partial_n u_0^\varepsilon \rangle_{\Gamma^\varepsilon_k} - \sum_{j\neq k} f_k^0(x^{(j)})\langle \1, \partial_n u_0^\varepsilon \rangle_{\Gamma^\varepsilon_j} +O\left(  |K^\varepsilon|^\infty \right)
\end{align}
where we also used \eqref{u,1} as well as Corollary \ref{corf0}, \eqref{defK}, and \eqref{crochetbord2}.
This precisely proves that \eqref{amq3} holds true.

Now, let $|\gamma|>1$ and suppose that the result holds true for all
$\beta \in \N^{N}$ such that $1\leq |\beta|\leq |\gamma|-1$ with $\beta_{k_0}=0$.
We do the same integration by parts as in the previous computation :
\begin{align*}
(\lambda_0^\varepsilon)^{|\gamma|} \langle \hat f^{\varepsilon,(k_0)}_\gamma, u_0^\varepsilon \rangle&=(\lambda_0^\varepsilon)^{|\gamma|-1}
 \langle \hat f^{\varepsilon,(k_0)}_\gamma, -\Delta u_0^\varepsilon \rangle\\
		&=(\lambda_0^\varepsilon)^{|\gamma|-1}
\Big( 
 \langle -\Delta \hat f^{\varepsilon,(k_0)}_\gamma,  u_0^\varepsilon \rangle -  \langle \hat f^{\varepsilon,(k_0)}_\gamma, \partial_n u_0^\varepsilon \rangle_{\Gamma^\varepsilon_{\mathcal D}}  \Big)
\\
		&=(\lambda_0^\varepsilon)^{|\gamma|-1} \Big(
 \hat a^{(k_0)}_\gamma \langle \1 ,  u_0^\varepsilon \rangle +\langle \hat F^{(k_0)}_\gamma, u_0^\varepsilon \rangle - \sum_{j=1}^N \langle \hat f^{\varepsilon,(k_0)}_\gamma, \partial_n u_0^\varepsilon \rangle_{\Gamma^\varepsilon_j}  \Big)
 \qquad \text{by \eqref{eq:fgamepd}
           }.
           \end{align*}
           Thus,
  \begin{align}
(\lambda_0^\varepsilon)^{|\gamma|} \langle \hat f^{\varepsilon,(k_0)}_\gamma, u_0^\varepsilon \rangle         
		&=(\lambda_0^\varepsilon)^{|\gamma|-1}\left(
 - \hat a^{(k_0)}_\gamma +\langle \hat F^{(k_0)}_\gamma, u_0^\varepsilon \rangle \right) \\  
		&\quad - (\lambda_0^\varepsilon)^{|\gamma|-1} \sum_{j=1}^N  \left( \hat v^{(k_0)}_\gamma(x^{(j)})  - \frac{\hat b^{(k_0)}_{\gamma,j}}{K_j^\varepsilon} \1_{\gamma_j\neq 0}  +\sum_{k\in \mathrm{supp}\, \gamma \backslash\{j\}} \hat b^{(k_0)}_{\gamma,k} f_k^0(x^{(j)})    \right)
\langle \1, \partial_n u_0^\varepsilon \rangle_{\Gamma^\varepsilon_j}\label{gamma-12}\\
		&\quad +O\left( |K^\varepsilon|^\infty \right) 
\end{align}
where we once again used \eqref{u,1} as well as Corollary \ref{corf0} and \eqref{crochetbord2}.
We are now going to see that
$(\lambda_0^\varepsilon)^{|\gamma|-1}\langle \hat F^{(k_0)}_\gamma,
u_0^\varepsilon \rangle$ only features some terms of order
$(\lambda_0^\varepsilon)^n$ with $n\leq |\gamma|-2$. If this is
proven, \eqref{gamma-1} is a direct consequence of~\eqref{gamma-12} by
considering the terms of order $(\lambda_0^\varepsilon)^n$, with $n >|\gamma|-2$.
It thus remains to compute
$(\lambda_0^\varepsilon)^{|\gamma|-1}\langle \hat F^{(k_0)}_\gamma,
u_0^\varepsilon \rangle$ and to show that \eqref{cn}, \eqref{dn} and \eqref{ln} hold true.
By \eqref{eq:F}, Lemma \ref{f0} and the induction hypothesis, we have
\begin{align}
&(\lambda_0^\varepsilon)^{|\gamma|-1}\langle \hat F^{(k_0)}_\gamma,
                u_0^\varepsilon \rangle  = \mathop{\sum_{\beta \in \N^{N};\,\beta\leq \gamma 
;}}_{\substack{1 \leq |\beta|\leq |\gamma|-1}} (\lambda_0^\varepsilon)^{|\gamma|-1-|\beta|}  \hat a^{(k_0)}_{\gamma-\beta} (\lambda_0^\varepsilon)^{|\beta|} \left\langle \hat f^{\varepsilon,(k_0)}_\beta +O\Big( |K^\varepsilon|^\infty \Big) , u_0^\varepsilon \right\rangle \\
			&=  \!\!\! \mathop{\sum_{\beta \in \N^{N};\,\beta\leq \gamma 
;}}_{\substack{1 \leq |\beta|\leq |\gamma|-1}} \!\!\! (\lambda_0^\varepsilon)^{|\gamma|-1-|\beta|}  \hat a^{(k_0)}_{\gamma-\beta}   
		 \sum_{n=0}^{|\beta|-1}(\lambda_0^\varepsilon)^{n} \left( c^{(k_0)}_n(\beta)+\sum_{k\in \mathrm{supp}\, \beta} d^{(k_0)}_{n,k}(\beta) \frac{\lambda_0^\varepsilon}{K_k^\varepsilon} Y^\varepsilon_k + \sum_{k=1}^N   \ell^{(k_0)}_{n,k}(\beta) \lambda_0^\varepsilon Y^\varepsilon_k  \right) +O\left( |K^\varepsilon|^\infty \right)\\
			&=\sum_{s=1}^{|\gamma|-1} \!\mathop{\sum_{\beta \in \N^{N};\,\beta\leq \gamma 
;}}_{\substack{|\beta|=s}} \!\!\! \hat a^{(k_0)}_{\gamma-\beta} \sum_{n=0}^{s-1}(\lambda_0^\varepsilon)^{n+|\gamma|-1-s} \left( c^{(k_0)}_n(\beta)+\sum_{k\in \mathrm{supp}\, \beta} d^{(k_0)}_{n,k}(\beta) \frac{\lambda_0^\varepsilon}{K_k^\varepsilon}  Y^\varepsilon_k + \sum_{k=1}^N   \ell^{(k_0)}_{n,k}(\beta) \lambda_0^\varepsilon Y^\varepsilon_k  \right) +O\left( |K^\varepsilon|^\infty \right).
\end{align}
By the change of indices $m=n+|\gamma|-1-s$, the right-hand side becomes
(up to the $O\left( |K^\varepsilon|^\infty \right)$ term)
\begin{align}
  &\sum_{s=1}^{|\gamma|-1}\!\!\! \mathop{\sum_{\beta \in \N^{N};\,\beta\leq \gamma 
;}}_{\substack{|\beta|=s}}  \!\!\! \hat a^{(k_0)}_{\gamma-\beta} \!\!\!
 \sum_{m=|\gamma|-1-s}^{|\gamma|-2}\!\!\! (\lambda_0^\varepsilon)^{m} \Bigg( c^{(k_0)}_{m-(|\gamma|-1-s)}(\beta) 
 +\sum_{k\in \mathrm{supp}\, \beta} d^{(k_0)}_{m-(|\gamma|-1-s),k}(\beta) \frac{\lambda_0^\varepsilon}{K_k^\varepsilon} Y^\varepsilon_k+ \sum_{k=1}^N   \ell^{(k_0)}_{m-(|\gamma|-1-s),k}(\beta) \lambda_0^\varepsilon Y^\varepsilon_k  \Bigg)\\
			&=\sum_{m=0}^{|\gamma|-2}(\lambda_0^\varepsilon)^{m}
                   \!\!\!\sum_{s=|\gamma|-1-m}^{|\gamma|-1}\mathop{\sum_{\beta \in \N^{N};\,\beta\leq \gamma 
;}}_{\substack{|\beta|=s}} \!\!\! \hat a^{(k_0)}_{\gamma-\beta} \Bigg( c^{(k_0)}_{m-(|\gamma|-1-s)}(\beta) 
 +\sum_{k\in \mathrm{supp}\, \beta} d^{(k_0)}_{m-(|\gamma|-1-s),k}(\beta) \frac{\lambda_0^\varepsilon}{K_k^\varepsilon} Y^\varepsilon_k+ \sum_{k=1}^N   \ell^{(k_0)}_{m-(|\gamma|-1-s),k}(\beta) \lambda_0^\varepsilon Y^\varepsilon_k  \Bigg)\\
			&=\sum_{m=0}^{|\gamma|-2}(\lambda_0^\varepsilon)^{m} \!\!\!
\mathop{\sum_{\beta \in \N^{N};\,\beta\leq \gamma 
;}}_{\substack{|\gamma|-1-m\leq |\beta|\leq |\gamma|-1}}  \!\!\!\hat a^{(k_0)}_{\gamma-\beta} \Bigg( c^{(k_0)}_{m-(|\gamma|-1-|\beta|)}(\beta) 
 +\sum_{k\in \mathrm{supp}\, \beta} d^{(k_0)}_{m-(|\gamma|-1-|\beta|),k}(\beta) \frac{\lambda_0^\varepsilon}{K_k^\varepsilon} Y^\varepsilon_k+ \sum_{k=1}^N   \ell^{(k_0)}_{m-(|\gamma|-1-|\beta|),k}(\beta)  \lambda_0^\varepsilon Y^\varepsilon_k  \Bigg)
\end{align}
where we also switched the orders of summation.
Identifying the different powers of $\lambda_0^\varepsilon$ allows to establish the formulas \eqref{cn}, \eqref{dn} and \eqref{ln} and the proof is thus complete.
\end{proof}

\begin{lem}\label{recd0}
Let $k_0\in \llbracket 1,N \rrbracket$ and $\gamma \in \N^N$ such that $\gamma_{k_0}=0$.
For all $k,j\in \llbracket 1,N \rrbracket \backslash \{k_0\} $, it holds 
$$d^{(k_0)}_{0,k}(\gamma+e_k)=d^{(k_0)}_{0,j}(\gamma+e_j).$$
Moreover, these coefficients are positive.
\end{lem}

\begin{proof}

Let us proceed by induction on $|\gamma|\geq 0$. When $|\gamma|=0$, the result clearly holds true since $d^{(k_0)}_{0,k}(e_k)=1>0$ thanks to \eqref{cdl0}.

Let $\gamma\in \N^N$ now such that $|\gamma|\geq 1$ and suppose that the result holds true for all $\beta \in \N^N$ with $|\beta|<|\gamma|$.
We have thanks to \eqref{dn} and the change of index $\alpha=\beta-e_k$
\begin{align}
d^{(k_0)}_{0,k}(\gamma+e_k)&=\mathop{\sum_{\beta \in \N^{N};\,\beta\leq \gamma+e_k 
; \, k\in \mathrm{supp}\, \beta}}_{\substack{|\beta|= |\gamma|}} \hat a^{(k_0)}_{\gamma+e_k-\beta}\, d^{(k_0)}_{0,k}(\beta)\\
		&=\mathop{\sum_{\alpha \in \N^{N};\,\alpha\leq \gamma 
}}_{\substack{|\alpha|= |\gamma|-1}} \hat a^{(k_0)}_{\gamma-\alpha}\, d^{(k_0)}_{0,k}(\alpha+e_k)\\
		&=\mathop{\sum_{\alpha \in \N^{N};\,\alpha\leq \gamma 
}}_{\substack{|\alpha|= |\gamma|-1}} \frac{\pi}{|\Omega|}\,
           d^{(k_0)}_{0,j}(\alpha+e_j)\label{eq:dpositive}\\
  &=d^{(k_0)}_{0,j}(\gamma+e_j)
\end{align}
by the induction hypothesis, where we used, in the last
equality, the fact that there exists $\ell$ such that $\hat
  a^{(k_0)}_{\gamma-\alpha}=\hat a^{(k_0)}_{\ell}$ and thus $\hat
  a^{(k_0)}_{\gamma-\alpha}=\frac{\pi}{|\Omega|}$ (see Lemma~\ref{lemcomp}). Notice that
  from~\eqref{eq:dpositive}, $d^{(k_0)}_{0,k}(\gamma+e_k)>0$ thanks to
  the induction hypothesis.
This completes the proof.
\end{proof}

\begin{lem}\label{syst}
For all $M\geq 1$, there exist a matrix $A_M^\varepsilon\in \mathcal M_N(\R)$ and a vector $B_M^\varepsilon\in \R^N$ 
such that 
\begin{enumerate}[label=\alph*)]
\item The vector $Y^\varepsilon=(Y^\varepsilon_{k_0})_{k_0\in \llbracket 1, N \rrbracket}$ solves the system $A_M^\varepsilon Y^\varepsilon=B_M^\varepsilon$. \label{a}
\item $A_M^\varepsilon$ and $B_M^\varepsilon$ both admit some expansion of the form
\begin{align}\label{devAetB}
A_M^\varepsilon = \sum_{m=M+1}^{2M} A^\varepsilon_{M,m} +O\left(|K^\varepsilon|^{2M+1}\right) \qquad \text{and} \qquad B_M^\varepsilon = \sum_{m=M+1}^{2M} B^\varepsilon_{M,m} +O\left(|K^\varepsilon|^{2M+1}\right)
\end{align}
where each entry of $A^\varepsilon_{M,m}$ and $B^\varepsilon_{M,m}$
belongs to $\mathscr P^{hom}_m(K^\varepsilon)$ (which is the set of
homogeneous polynomial of $(K^\varepsilon_1,\ldots, K^\varepsilon_N)$ of degree $m$). \label{b}
\item The leading terms (namely the ones homogeneous with degree $M+1$) are
\begin{equation}\label{AMprem}
  \left(A^\varepsilon_{M,M+1}\right)_{k_0,k}=\left(  \frac{\pi}{|\Omega|}\overline{K^\varepsilon} \right)^{M+1}  \delta_{k_0,k} + \frac{\pi}{|\Omega|} K^\varepsilon_{k_0}\mathop{\sum_{\gamma \in \N^{N};\gamma_{k_0}=0, \gamma_k\neq 0}}_{\substack{1\leq |\gamma|\leq M}} (K^\varepsilon)^{\gamma-e_k} \left(\frac{\pi}{|\Omega|}\overline{K^\varepsilon}\right)^{M+1-|\gamma|}  d^{(k_0)}_{0,k}(\gamma)
\end{equation}
and
\begin{align}\label{BMprem}
\left( B^\varepsilon_{M,M+1} \right)_{k_0}= \frac{\pi}{|\Omega|} K^\varepsilon_{k_0}\left( \left(  \frac{\pi}{|\Omega|}\overline{K^\varepsilon} \right)^M -\mathop{\sum_{\gamma \in \N^{N};\gamma_{k_0}=0}}_{\substack{1\leq |\gamma|\leq M}}  (K^\varepsilon)^\gamma \left(\frac{\pi}{|\Omega|}\overline{K^\varepsilon}\right)^{M-|\gamma|}  c^{(k_0)}_0(\gamma)  \right) .  
\end{align}
 \label{c}
\end{enumerate}
\end{lem}

Here and in the following, $\mathcal M_N(\R)$ denotes the set of $N
\times N$ matrices with real entries. Notice that in~\eqref{AMprem}, the second term is zero for $k=k_0$
(since the sum is over an empty set).  
\begin{proof}
Let $k_0\in \llbracket 1,N \rrbracket$.
According to Proposition \ref{propphi}, Corollary \ref{corf0} and \eqref{uholder}, the function $\hat \varphi^\varepsilon_{k_0}$ satisfies 
\begin{align}\label{devhatphi}
\hat\varphi^\varepsilon_{k_0}
& =1+O(|K^\varepsilon|) \qquad \qquad \qquad \text{in }L^2(\Omega),
\end{align}
and
\begin{align}\label{devhatphi2}
\hat\varphi^\varepsilon_{k_0} 
& =\hat \alpha_{k_0}^\varepsilon \1_{\Gamma_{k_0}^\varepsilon} +O(|K^\varepsilon|^M) \qquad \text{uniformly on }\Gamma_{\mathcal D}^\varepsilon,
\end{align}
where the constant $\hat \alpha_{k_0}^\varepsilon$ is defined by
\begin{align}\label{hatalpha}
\hat \alpha_{k_0}^\varepsilon= 1+  \mathop{\sum_{\gamma \in \N^{N}; \gamma_{k_0}=0}}_{\substack{1\leq |\gamma|\leq M}} (K^\varepsilon)^\gamma \hat f_\gamma^{0,(k_0)}(x^{(k_0)}).
\end{align}
Using also \eqref{crochetbord2}, this allows to write
\begin{align}
Y^\varepsilon_{k_0}
		&=\frac{1}{\lambda_0^\varepsilon \, \hat\alpha_{k_0}^\varepsilon}\langle \partial_n u_0^\varepsilon , \hat\varphi_{k_0}^\varepsilon \rangle_{\Gamma_{\mathcal D}^\varepsilon}+O\left( |K^\varepsilon|^M \right) \qquad \text{using \eqref{lambdaesp} for the error term}\\
		&=\frac{1}{\lambda_0^\varepsilon \, \hat\alpha_{k_0}^\varepsilon}  \Big(  \langle u_0^\varepsilon , -\Delta \hat\varphi_{k_0}^\varepsilon \rangle - \langle -\Delta  u_0^\varepsilon ,  \hat\varphi_{k_0}^\varepsilon \rangle    \Big)+O\left( |K^\varepsilon|^M \right)\qquad \text{by Proposition \ref{propphi} item \ref{phiNeum}}\\
		&=\frac{\hat\lambda_{0,k_0}^\varepsilon- \lambda_0^\varepsilon}{\lambda_0^\varepsilon \, \hat\alpha_{k_0}^\varepsilon}   \langle u_0^\varepsilon ,  \hat\varphi_{k_0}^\varepsilon \rangle +O\Big( |K^\varepsilon|^{M} \Big) \qquad \text{by Proposition \ref{propphi} item \ref{phiDelta}} \\
		&=\frac{\hat \lambda_{0,k_0}^\varepsilon- \tilde \lambda_0^\varepsilon}{\lambda_0^\varepsilon \, \hat\alpha_{k_0}^\varepsilon}   \langle u_0^\varepsilon ,  \hat\varphi_{k_0}^\varepsilon \rangle +O\Big( |K^\varepsilon|^{M} \Big)\label{calcX}
\end{align}
where the last equality relies on Theorem \ref{lam0} and Lemma \ref{phi,u}.
Now, thanks to \eqref{u,1} and Lemma \ref{cdl}, we have
\begin{align}
&   (\lambda_0^\varepsilon)^M           \langle u_0^\varepsilon ,  \hat\varphi_{k_0}^\varepsilon\rangle=-(\lambda_0^\varepsilon)^M+ \mathop{\sum_{\gamma \in \N^{N};\gamma_{k_0}=0}}_{\substack{1\leq |\gamma|\leq M}} (K^\varepsilon)^\gamma (\lambda_0^\varepsilon)^{M-|\gamma|} (\lambda_0^\varepsilon)^{|\gamma|} \langle  \hat f^{\varepsilon,(k_0)}_\gamma ,  u_0^\varepsilon \rangle \\
		&=-(\lambda_0^\varepsilon)^M+ \mathop{\sum_{\gamma \in \N^{N};\gamma_{k_0}=0}}_{\substack{1\leq |\gamma|\leq M}} (K^\varepsilon)^\gamma (\lambda_0^\varepsilon)^{M-|\gamma|}  \sum_{n=0}^{|\gamma|-1}(\lambda_0^\varepsilon)^{n} \left( c^{(k_0)}_n(\gamma)+\sum_{k\in \mathrm{supp}\, \gamma} d^{(k_0)}_{n,k}(\gamma) \frac{\lambda_0^\varepsilon}{K_k^\varepsilon} Y^\varepsilon_k + \sum_{k=1}^N   \ell^{(k_0)}_{n,k}(\gamma) \lambda_0^\varepsilon Y^\varepsilon_k  \right)+O\left( |K^\varepsilon|^\infty \right)\\
		&=-(\lambda_0^\varepsilon)^M+ \mathop{\sum_{\gamma \in \N^{N};\gamma_{k_0}=0}}_{\substack{1\leq |\gamma|\leq M}} \sum_{n=0}^{|\gamma|-1} (K^\varepsilon)^\gamma (\lambda_0^\varepsilon)^{M+n-|\gamma|} \,  c^{(k_0)}_n(\gamma) + \mathop{\sum_{\gamma \in \N^{N};\gamma_{k_0}=0}}_{\substack{1\leq |\gamma|\leq M}} \sum_{n=0}^{|\gamma|-1} (K^\varepsilon)^\gamma (\lambda_0^\varepsilon)^{M+n+1-|\gamma|}  \, \sum_{k=1}^N \ell^{(k_0)}_{n,k}(\gamma) Y^\varepsilon_k\\
		& \quad +  \mathop{\sum_{\gamma \in \N^{N};\gamma_{k_0}=0}}_{\substack{1\leq |\gamma|\leq M}} \sum_{n=0}^{|\gamma|-1}  \sum_{k\in \mathrm{supp}\, \gamma}  (K^\varepsilon)^{\gamma-e_k} (\lambda_0^\varepsilon)^{M+n+1-|\gamma|} \,  d^{(k_0)}_{n,k}(\gamma) Y^\varepsilon_k +O\left( |K^\varepsilon|^\infty \right)\\
		&=-(\lambda_0^\varepsilon)^M+ \mathop{\sum_{\gamma \in \N^{N};\gamma_{k_0}=0}}_{\substack{1\leq |\gamma|\leq M}} \sum_{n=0}^{|\gamma|-1} (K^\varepsilon)^\gamma (\lambda_0^\varepsilon)^{M+n-|\gamma|} \,  c^{(k_0)}_n(\gamma) +\sum_{k=1}^N \left( \mathop{\sum_{\gamma \in \N^{N};\gamma_{k_0}=0}}_{\substack{1\leq |\gamma|\leq M}} \sum_{n=0}^{|\gamma|-1} (K^\varepsilon)^\gamma (\lambda_0^\varepsilon)^{M+n+1-|\gamma|}  \,  \ell^{(k_0)}_{n,k}(\gamma) \right) Y^\varepsilon_k\\
		& \quad + \sum_{k=1}^N \left(  \mathop{\sum_{\gamma \in \N^{N};\gamma_{k_0}=0, \gamma_k\neq 0}}_{\substack{1\leq |\gamma|\leq M}} \sum_{n=0}^{|\gamma|-1}   (K^\varepsilon)^{\gamma-e_k} (\lambda_0^\varepsilon)^{M+n+1-|\gamma|} \,  d^{(k_0)}_{n,k}(\gamma)\right) Y^\varepsilon_k +O\left( |K^\varepsilon|^\infty \right). \label{u,hatphi}
\end{align}
Putting \eqref{calcX} and \eqref{u,hatphi} together shows that item \ref{a} holds true by setting the $(k_0,k)$ entry of $A_M^\varepsilon$ (resp. the $k_0$ entry of $B_M^\varepsilon$) to
\begin{align}\label{entreeA}
\text{ }\qquad  (\lambda_0^\varepsilon)^{M+1} \hat\alpha_{k_0}^\varepsilon \delta_{k_0,k}- \left(\hat \lambda_{0,k_0}^\varepsilon- \tilde \lambda_0^\varepsilon\right)& \left(  \mathop{\sum_{\gamma \in \N^{N};\gamma_{k_0}=0}}_{\substack{1\leq |\gamma|\leq M}} \sum_{n=0}^{|\gamma|-1} (K^\varepsilon)^\gamma (\lambda_0^\varepsilon)^{M+n+1-|\gamma|}  \,  \ell^{(k_0)}_{n,k}(\gamma)  \right. \\
			& \qquad  \left.+   \mathop{\sum_{\gamma \in \N^{N};\gamma_{k_0}=0, \gamma_k\neq 0}}_{\substack{1\leq |\gamma|\leq M}} \sum_{n=0}^{|\gamma|-1}   (K^\varepsilon)^{\gamma-e_k} (\lambda_0^\varepsilon)^{M+n+1-|\gamma|} \,  d^{(k_0)}_{n,k}(\gamma) \right)
\end{align}
resp.
\begin{align}\label{entreeB}
\left(\hat \lambda_{0,k_0}^\varepsilon- \tilde \lambda_0^\varepsilon\right)\left(  -(\lambda_0^\varepsilon)^M+ \mathop{\sum_{\gamma \in \N^{N};\gamma_{k_0}=0}}_{\substack{1\leq |\gamma|\leq M}} \sum_{n=0}^{|\gamma|-1} (K^\varepsilon)^\gamma (\lambda_0^\varepsilon)^{M+n-|\gamma|} \,  c^{(k_0)}_n(\gamma)  \right).
\end{align}
Besides, we have thanks to \eqref{hatalpha} and Theorem \ref{lam0} that $\hat\alpha_{k_0}^\varepsilon$ admits an expansion of order $M$ and that $\lambda_0^\varepsilon$ admits an expansion whose leading term is of degree 1.
Using also \eqref{hatlam}, this shows that \eqref{entreeA} admits an expansion of order $2M$ with leading term of order $M+1$ and that, similarly \eqref{entreeB} admits an expansion of order $2M$ with leading term of order $M+1$.
This proves item \ref{b}.\\
Finally, using \eqref{hatalpha} and Theorem \ref{lam0} again, we obtain that the leading term of $\lambda_0^\varepsilon$ is
\begin{align}\label{lam0prem}
\frac{\pi}{|\Omega|}\overline{K^\varepsilon},
\end{align}
the one of $\hat\alpha_{k_0}^\varepsilon$ is simply $1$, the one of $\hat \lambda_{0,k_0}^\varepsilon- \tilde \lambda_0^\varepsilon$ is 
$$-\frac{\pi}{|\Omega|} K^\varepsilon_{k_0}$$
while the one of
$$\mathop{\sum_{\gamma \in \N^{N};\gamma_{k_0}=0}}_{\substack{1\leq |\gamma|\leq M}} \sum_{n=0}^{|\gamma|-1} (K^\varepsilon)^\gamma (\lambda_0^\varepsilon)^{M+n+1-|\gamma|}  \,  \ell^{(k_0)}_{n,k}(\gamma)  +   \mathop{\sum_{\gamma \in \N^{N};\gamma_{k_0}=0, \gamma_k\neq 0}}_{\substack{1\leq |\gamma|\leq M}} \sum_{n=0}^{|\gamma|-1}   (K^\varepsilon)^{\gamma-e_k} (\lambda_0^\varepsilon)^{M+n+1-|\gamma|} \,  d^{(k_0)}_{n,k}(\gamma)$$
is obtained using \eqref{lam0prem} by keeping only the terms of degree $M$ and amounts to
$$ \mathop{\sum_{\gamma \in \N^{N};\gamma_{k_0}=0, \gamma_k\neq 0}}_{\substack{1\leq |\gamma|\leq M}} (K^\varepsilon)^{\gamma-e_k} \left(\frac{\pi}{|\Omega|}\overline{K^\varepsilon}\right)^{M+1-|\gamma|} \,  d^{(k_0)}_{0,k}(\gamma).$$
This yields \eqref{AMprem}.
The expression \eqref{BMprem} can be established in a very similar fashion, allowing to conclude the proof of item \ref{c} and of the Lemma.
\end{proof}
\hop
As one can expect, the invertibility of the matrix $A^\varepsilon_{M,M+1}$ (and thus the value of its determinant) will play a crucial role in solving the system
$$A_M^\varepsilon Y^\varepsilon=B_M^\varepsilon.$$
We will denote
$$p_M(K^\varepsilon)= \det A^\varepsilon_{M,M+1}\in \mathscr P^{hom}_{(M+1)N}(K^\varepsilon)$$
and we also define $\tilde d^{(k_0)}$ by 
\begin{align}\label{dtilde}
\forall \beta\in \N^N \text{ with } \beta_{k_0}=0, \qquad \tilde d^{(k_0)} (\beta) =  d^{(k_0)}_{0,k} (\beta+e_k) 
\end{align}
which is well defined, independent of $k\neq k_0$, and positive thanks to Lemma \ref{recd0}:
\begin{align}\label{dtildepos}
\tilde d^{(k_0)} (\beta) >0.
\end{align}

\begin{lem}\label{calcdet}
Let us recall that $p_M(K^\varepsilon)= \det A^\varepsilon_{M,M+1}\in
\mathscr P^{hom}_{(M+1)N}(K^\varepsilon)$. We have
\begin{align}\label{pM}
p_M(K^\varepsilon)&= \left( \frac{\pi}{|\Omega|} \overline{K^\varepsilon} \right)^N  \prod_{k_0= 1}^N  \mathop{\sum_{\gamma \in \N^{N};\gamma_{k_0}=0}}_{\substack{ |\gamma|= M}}
 \tilde d^{(k_0)}(\gamma) (K^\varepsilon)^{\gamma} \\
		&\quad +  \left(\frac{\pi}{|\Omega|}\right)^{N+1}    \overline{K^\varepsilon}^N \sum_{k_0= 1}^N   \mathop{\sum_{\gamma \in \N^{N};\gamma_{k_0}=1}}_{\substack{1\leq |\gamma|\leq M}} (K^\varepsilon)^{\gamma} \left(\frac{\pi}{|\Omega|}\overline{K^\varepsilon}\right)^{M-|\gamma|}  
 \tilde d^{(k_0)}
(\gamma-e_{k_0})  \prod_{k\neq k_0}  \mathop{\sum_{\gamma \in \N^{N};\gamma_{k}=0}}_{\substack{ |\gamma|= M}}
 \tilde d^{(k)}(\gamma) (K^\varepsilon)^{\gamma}.
\end{align}
\end{lem}

\begin{proof}
We are first going to show that all the entries of $A^\varepsilon_{M,M+1}$ which are on the same line and outside the diagonal are equal.
Indeed, let $k_0
\in \llbracket 1,N \rrbracket$ and $k \in \llbracket 1,N \rrbracket \backslash \{k_0\}$.
We have thanks to \eqref{AMprem} and the successive changes of indices $\beta=\gamma-e_k$  and $\gamma=\beta+e_{k_0}$
\begin{align}
\left( A^\varepsilon_{M,M+1}\right)_{k_0,k}
&=\frac{\pi}{|\Omega|} K^\varepsilon_{k_0}
  \mathop{\sum_{\gamma \in \N^{N};\gamma_{k_0}=0, \gamma_k\neq 0}}_{\substack{1\leq |\gamma|\leq M}} (K^\varepsilon)^{\gamma-e_k} \left(\frac{\pi}{|\Omega|}\overline{K^\varepsilon}\right)^{M+1-|\gamma|}  d^{(k_0)}_{0,k}(\gamma)
\\
		&=\frac{\pi}{|\Omega|}  \mathop{\sum_{\beta \in \N^{N};\beta_{k_0}=0}}_{\substack{0\leq |\beta|\leq M-1}} (K^\varepsilon)^{\beta+e_{k_0}} \left(\frac{\pi}{|\Omega|}\overline{K^\varepsilon}\right)^{M-|\beta|}  
 d^{(k_0)}_{0,k}(\beta+e_k) 
\\
		&=\frac{\pi}{|\Omega|}  \mathop{\sum_{\gamma \in \N^{N};\gamma_{k_0}=1}}_{\substack{1\leq |\gamma|\leq M}} (K^\varepsilon)^{\gamma} \left(\frac{\pi}{|\Omega|}\overline{K^\varepsilon}\right)^{M+1-|\gamma|}  
 \tilde d^{(k_0)}
(\gamma-e_{k_0}) \label{=horsdiag}
\end{align}
where $\tilde d^{(k_0)}$ is defined in \eqref{dtilde}.
It will now be more convenient to work with $(A^\varepsilon_{M,M+1})^\top$ rather than $A^\varepsilon_{M,M+1}$ itself.
The identity \eqref{=horsdiag} shows that $(A^\varepsilon_{M,M+1})^\top$ can be decomposed as 
\begin{align}\label{decAt}
(A^\varepsilon_{M,M+1})^\top=D_M^\varepsilon+\tilde A^\varepsilon_{M}
\end{align}
where $D_M^\varepsilon$ is the diagonal matrix
\begin{align}\label{DM}
\text{ }\qquad D_M^\varepsilon= \frac{\pi}{|\Omega|}\overline{K^\varepsilon}\; \mathrm{diag} \left( \left(  \frac{\pi}{|\Omega|}\overline{K^\varepsilon} \right)^{M}- \frac{\pi}{|\Omega|}  \mathop{\sum_{\gamma \in \N^{N};\gamma_{k_0}=1}}_{\substack{1\leq |\gamma|\leq M}} (K^\varepsilon)^{\gamma} \left(\frac{\pi}{|\Omega|}\overline{K^\varepsilon}\right)^{M-|\gamma|}  
 \tilde d^{(k_0)}
(\gamma-e_{k_0})  \right)_{1\leq k_0\leq N}
\end{align}
and $\tilde A_M^\varepsilon$ is a rank 1 matrix which can be represented as
\begin{align}\label{1LM}
\tilde A_M^\varepsilon=\begin{pmatrix} (L_M^\varepsilon)^\top\\ \vdots \\ (L_M^\varepsilon)^\top
\end{pmatrix}=\1 (L_M^\varepsilon)^\top
\end{align}
with $L_M^\varepsilon\in \R^N$ given by
\begin{align}\label{LM}
L_M^\varepsilon= \left(\frac{\pi}{|\Omega|}  \mathop{\sum_{\gamma \in \N^{N};\gamma_{k_0}=1}}_{\substack{1\leq |\gamma|\leq M}} (K^\varepsilon)^{\gamma} \left(\frac{\pi}{|\Omega|}\overline{K^\varepsilon}\right)^{M+1-|\gamma|}  
 \tilde d^{(k_0)}
(\gamma-e_{k_0})  \right)_{1\leq k_0\leq N}\neq 0,
\end{align}
since all the components of $L_M^\varepsilon$ are positive
(see~\eqref{dtildepos}).
We have thanks to Lemma \ref{Dinv} that $D_M^\varepsilon$ is invertible and thus \eqref{decAt} and \eqref{1LM} yield
\begin{align}\label{decAt2}
(A^\varepsilon_{M,M+1})^\top=D_M^\varepsilon \hat A^\varepsilon_{M}
\end{align}
where 
$$\hat A^\varepsilon_{M}= \mathrm{Id}+ (D_M^\varepsilon)^{-1} \, \1 \, (L_M^\varepsilon)^\top .$$
Clearly, we have
$$\hat A_M^\varepsilon|_{(L_M^\varepsilon)^\perp}=\mathrm{Id}|_{(L_M^\varepsilon)^\perp}$$
and therefore 
\begin{align}
\det \hat A_M^\varepsilon&= \hat A_M^\varepsilon \frac{L_M^\varepsilon}{|L_M^\varepsilon|} \cdot \frac{L_M^\varepsilon}{|L_M^\varepsilon|}=1+(D_M^\varepsilon)^{-1} \1 \cdot L_M^\varepsilon\\
		&=1+\frac{\pi}{|\Omega|}\sum_{k_0= 1}^N   \left( \mathop{\sum_{\gamma \in \N^{N};\gamma_{k_0}=0}}_{\substack{ |\gamma|= M}}
 \tilde d^{(k_0)}(\gamma) (K^\varepsilon)^{\gamma}\right)^{-1}  \mathop{\sum_{\gamma \in \N^{N};\gamma_{k_0}=1}}_{\substack{1\leq |\gamma|\leq M}} (K^\varepsilon)^{\gamma} \left(\frac{\pi}{|\Omega|}\overline{K^\varepsilon}\right)^{M-|\gamma|}  
 \tilde d^{(k_0)}
(\gamma-e_{k_0})
\end{align}
which we computed using Lemma \ref{Dinv} again.
Thanks to \eqref{decAt2} and Lemma \ref{Dinv} once more, we finally deduce
that \eqref{pM} holds true.
\end{proof}

\begin{lem}\label{minodet}
For all $M\geq 1$, there exists $C_M>0$ such that the determinant of $A^\varepsilon_{M,M+1}$ satisfies
$$p_M(K^\varepsilon)\geq \frac{|K^\varepsilon|^{(M+1)N}}{C_M}.$$
\end{lem}
\begin{proof}
Notice that \eqref{dtildepos} allows to remove some terms from \eqref{pM} to obtain an estimate from below of $p_M(K^\varepsilon)$.
We can thus write
\begin{align}
p_M(K^\varepsilon)&\geq \left( \frac{\pi}{|\Omega|} \right)^{N+1} \overline{K^\varepsilon}^{N}  \sum_{k_0= 1}^N   K^\varepsilon_{k_0}\left(\frac{\pi}{|\Omega|}\overline{K^\varepsilon}\right)^{M-1}   \prod_{k\neq k_0}  
 \tilde d^{(k)}(M e_{k_0}) (K^\varepsilon_{k_0})^M  \\
		&\geq \frac{1}{C} \overline{K^\varepsilon}^{N+M-1}  \sum_{k_0= 1}^N   (K^\varepsilon_{k_0})^{M(N-1)+1} \\
	 &\geq \frac{|K^\varepsilon|^{(M+1)N}}{C_M}
\end{align}
which is the desired result.
\end{proof}

Let us now state the second main result of this work.
\begin{thm}\label{thm:exit_pt}
Recall that $p_M(K^\varepsilon)$ was computed in \eqref{pM}.
For all $k_0\in \llbracket 1,N \rrbracket$, $M\geq 1$ 
and all integers $m,q$ such that $0\leq m \leq q \leq M-1$, there exists $Q_{M,m,q}^{(k_0)}\in \mathscr P^{hom}_{(M+1)N(m+1)+q}$ which can be computed thanks to the expansions \eqref{devAetB} and such that
\begin{align}\label{expY}
Y^\varepsilon_{k_0}= \sum_{q=0}^{M-1} \sum_{m=0}^{q}\frac{ Q_{M,m,q}^{(k_0)}(K^\varepsilon) }{p_M(K^\varepsilon)^{m+1}}+O\left( |K^\varepsilon|^M \right).
\end{align}
In particular, thanks to Lemma \ref{minodet}, the term of index $q$ satisfies
$$\sum_{m=0}^{q}\frac{ Q_{M,m,q}^{(k_0)}(K^\varepsilon) }{p_M(K^\varepsilon)^{m+1}}=O(|K^\varepsilon|^q).$$
Moreover, the term of index $q=0$ simplifies as 
\begin{align}\label{Yprem}
\frac{Q_{M,0,0}^{(k_0)}(K^\varepsilon)}{p_M(K^\varepsilon)}=\frac{K^\varepsilon_{k_0}}{\overline{K^\varepsilon}}.
\end{align}
\end{thm}
This result shows that there exists
  an expansion of the probability to leave through each of the exit
  windows $(\Gamma^\varepsilon_{k_0})_{1 \le k_0 \le N}$ in terms of
  powers of $K^\varepsilon$. As will be clear in the proof, it is not
  possible to give more explicit formulas for the terms in the
  expansion when $q\ge 1$, in full generality. This result and the
  proof below should thus be seen as a
  recipe to obtain such expansions.

\begin{proof}
Let us denote $\check A_M^\varepsilon$ the transpose of the cofactor matrix of $A^\varepsilon_{M,M+1}$ so that 
$$(A^\varepsilon_{M,M+1})^{-1}=\frac{\check A_M^\varepsilon}{p_M(K^\varepsilon)}.$$
Note that by \eqref{devAetB}, 
\begin{align}\label{degcheckA}
\text{the entries of }\check A_M^\varepsilon \text{ belong to }\mathscr P^{hom}_{(M+1)(N-1)}(K^\varepsilon).
\end{align}
Thanks to Lemma \ref{minodet}, we thus have
\begin{align}\label{estiminvA}
(A^\varepsilon_{M,M+1})^{-1}=O\left(  |K^\varepsilon|^{-(M+1)} \right).
\end{align}
Writing
$$A_M^\varepsilon=A^\varepsilon_{M,M+1}\Big( \mathrm{Id} + (A^\varepsilon_{M,M+1})^{-1} \left(  A_M^\varepsilon - A^\varepsilon_{M,M+1} \right) \Big) $$
and also using \eqref{devAetB}, this implies that $A_M^\varepsilon$ is
invertible for sufficiently small $\varepsilon$ and
\begin{align}
(A_M^\varepsilon)^{-1}&=\sum_{m=0}^{M-1} \left( - (A^\varepsilon_{M,M+1})^{-1} \left(  A_M^\varepsilon - A^\varepsilon_{M,M+1} \right) \right)^m (A^\varepsilon_{M,M+1})^{-1}  +O\left(  |K^\varepsilon|^{-1} \right) \label{A-1}\\
		&=\sum_{m=0}^{M-1} \frac{\left( - \check A^\varepsilon_{M} \left(  A_M^\varepsilon - A^\varepsilon_{M,M+1} \right) \right)^m \check A^\varepsilon_{M}}{p_M(K^\varepsilon)^{m+1}} +O\left(  |K^\varepsilon|^{-1} \right).
\end{align}
It follows by Lemma \ref{syst} that
\begin{align}\label{Y=}
Y^\varepsilon=\sum_{m=0}^{M-1} \frac{\left( - \check A^\varepsilon_{M} \left(  A_M^\varepsilon - A^\varepsilon_{M,M+1} \right) \right)^m \check A^\varepsilon_{M}}{p_M(K^\varepsilon)^{m+1}} B_M^\varepsilon +O\left(  |K^\varepsilon|^{M} \right).
\end{align}
For $\kappa$, $\varrho\in \N^*$ with $\kappa < \varrho$, let us denote
$$\mathcal E(\kappa; \varrho)=\left\{E^\varepsilon\in \mathcal
  M_N(\R)\text{ s.t. }   E^\varepsilon=\sum_{n=\kappa}^{\varrho-1} E^\varepsilon_n
  +O\left(|K^\varepsilon|^\varrho\right)\text{ with }
    \forall n\in \llbracket  \kappa, \varrho-1\rrbracket ,
  E^\varepsilon_n\in \mathcal M_N\left( \mathscr
    P^{hom}_n(K^\varepsilon) \right) 
\right\}$$
the set of matrix-valued functions of $K^\varepsilon$ admitting an
expansion of order $\varrho$ and with leading term of degree~$\kappa$ (here and in the following,
  $\mathcal M_N\left( \mathscr
    P^{hom}_n(K^\varepsilon) \right)$ denotes a $N\times N$ matrix
  with each entry in $ \mathscr
  P^{hom}_n(K^\varepsilon)$).
To prove the first statement, it suffices thanks to \eqref{Y=} to show that for all $0\leq m\leq M-1$, 
\begin{align}\label{bonneexp}
\left( - \check A^\varepsilon_{M} \left(  A_M^\varepsilon - A^\varepsilon_{M,M+1} \right) \right)^m \check A^\varepsilon_{M} B_M^\varepsilon \in \mathcal E\Big((M+1)N(m+1)+m\, ; \,(M+1)N(m+1)+M\Big).
\end{align}
Indeed, if \eqref{bonneexp} is established, then for all $q\in \llbracket m,  M-1\rrbracket$, there exists a column matrix $Q_{M,m,q}$ whose entries belong to $\mathscr P^{hom}_{(M+1)N(m+1)+q}$ and such that 
$$\left( - \check A^\varepsilon_{M} \left(  A_M^\varepsilon - A^\varepsilon_{M,M+1} \right) \right)^m \check A^\varepsilon_{M} B_M^\varepsilon= \sum_{q=m}^{M-1}Q_{M,m,q}+O\left( |K^\varepsilon|^{(M+1)N(m+1)+M} \right).$$
Combined with \eqref{Y=} and Lemma \ref{minodet}, this does show (after switching the orders of summation) that \eqref{expY} holds true.
It thus remains to establish \eqref{bonneexp} to obtain \eqref{expY}.
We know from \eqref{devAetB} that
$$A_M^\varepsilon - A^\varepsilon_{M,M+1}\in \mathcal E\Big(M+2\, ; \,2M+1\Big).$$
Using \eqref{degcheckA}, this gives
$$ - \check A^\varepsilon_{M} \left(  A_M^\varepsilon - A^\varepsilon_{M,M+1} \right)\in \mathcal E\Big((M+1)N+1\, ; \, (M+1)N+M\Big).$$
We thus have for all $0\leq m\leq M-1$,
$$\left( - \check A^\varepsilon_{M} \left(  A_M^\varepsilon -
    A^\varepsilon_{M,M+1} \right) \right)^m \in \mathcal
E\Big((M+1)Nm+m\, ; \,(M+1)Nm+M+m-\1_{m\geq1}\Big).$$
Using \eqref{devAetB} and \eqref{degcheckA} again, we finally obtain
$$\check A^\varepsilon_{M} B_M^\varepsilon  \in \mathcal E\Big((M+1)N\, ; \,(M+1)N+M\Big)$$
from which we deduce
$$\left( - \check A^\varepsilon_{M} \left(  A_M^\varepsilon - A^\varepsilon_{M,M+1} \right) \right)^m \check A^\varepsilon_{M} B_M^\varepsilon \in \mathcal E\Big((M+1)N(m+1)+m\, ; \,(M+1)N(m+1)+M+m-\1_{m\geq1}\Big)$$
and \eqref{bonneexp} follows as $m-\1_{m\geq1}\geq 0$.
The proof of \eqref{expY} is now complete.

\hop

For the last statement, i.e \eqref{Yprem}, it suffices to show that
\begin{align}\label{amqfin}
Y^\varepsilon_{k_0}= \frac{K^\varepsilon_{k_0}}{\overline{K^\varepsilon}} +O(|K^\varepsilon|).
\end{align}
Indeed, this would imply that
$$\frac{Q_{M,0,0}^{(k_0)}(K^\varepsilon)}{p_M(K^\varepsilon)}-\frac{K^\varepsilon_{k_0}}{\overline{K^\varepsilon}} =O(|K^\varepsilon|)$$
and consequently
$$\overline{K^\varepsilon}Q_{M,0,0}^{(k_0)}(K^\varepsilon)-K^\varepsilon_{k_0}p_M(K^\varepsilon)=O\left( |K^\varepsilon|^{(M+1)N+2} \right)$$
which allows to conclude thanks to Lemma \ref{multipoly}.
A 
quick way to establish \eqref{amqfin} is to use \eqref{calcX} and approximate each quantity on the RHS of this equation by its leading term.
Thanks to Lemma \ref{phi,u}, Theorem \ref{lam0} and \eqref{devhatphi}, this gives
$$Y^\varepsilon_{k_0}=\frac{-K^\varepsilon_{k_0}+O(|K^\varepsilon|^2)}{\left(\overline{K^\varepsilon}+O(|K^\varepsilon|^2)\right)(1+O(|K^\varepsilon|))}\left(-1+O(|K^\varepsilon|)\right) =\frac{K^\varepsilon_{k_0}}{\overline{K^\varepsilon}} +O(|K^\varepsilon|)$$
which is precisely \eqref{amqfin}.
The proof of the Theorem is therefore complete.
\end{proof}

\appendix

\section{The distribution $\partial_n
  u_0^\varepsilon/\lambda_0^\varepsilon$ is a probability measure}\label{sec:prob_res}

The objective of this section is to prove that one can identify the
normal derivative $\partial_n u_0^\varepsilon/\lambda_0^\varepsilon$ as a probability
measure on $\partial \Omega$, with support in $\overline{\Gamma^\varepsilon_{\mathcal D}}$.

From the regularity on the first eigenvector $u_0^\varepsilon$, the
normal derivative is well defined as a distribution on $\partial
\Omega$ using the standard $H^{-1/2}(\partial \Omega),
H^{1/2}(\partial \Omega)$ duality as follows: for any $\phi \in
H^{1/2}(\partial \Omega)$,
\begin{equation}\label{eq:weak_normal_derivative}
  \langle \partial_n u_0^\varepsilon, \phi \rangle_{\partial \Omega} = \int_{\Omega} \nabla u_0^\varepsilon \cdot
  \nabla \phi + \int_{\Omega} \Delta u_0^\varepsilon \phi.
  \end{equation}
Indeed, $u_0^\varepsilon \in H^1(\Omega)$, $\Delta u_0^\varepsilon \in
L^2(\Omega)$ and any function $\phi \in
H^{1/2}(\partial \Omega)$ can be extended to a function $\phi \in
H^1(\Omega)$ by standard results on the trace
operator~\cite{Evans,lions2012non} (and the right-hand-side does not
depend on this extension). 
  We now claim that the distribution $\partial_n u_0^\varepsilon$ can
  be identified as a positive Radon measure with support in
  $\overline{\Gamma^\varepsilon_{\mathcal D}}$, and with finite mass.

  \begin{prop}\label{prop:borne_dnu0}
This distribution $\partial_n u_0^\varepsilon$ naturally defined in
$H^{-1/2}(\partial \Omega)$ by~\eqref{eq:weak_normal_derivative}
can be identified as a positive Radon measure on $\partial \Omega$
with support in $\overline{\Gamma^\varepsilon_{\mathcal D}}$. Moreover, its total
mass is $\langle \partial_n u_0^\varepsilon , 1 \rangle_{\partial \Omega} =
\lambda_0^\varepsilon$. In particular, for any function $\psi \in
L^\infty(\Gamma^\varepsilon_{\mathcal D})$, one has
\begin{equation}\label{eq:dnu_mes_proba}
|\langle \partial_n u_0^\varepsilon, \psi\rangle_{\partial \Omega}| \le
  \lambda_0^\varepsilon \|\psi\|_{L^\infty(\Gamma^\varepsilon_{\mathcal D})}.
  \end{equation}
In particular, if $\psi=\alpha \1_{\Gamma_k^\varepsilon}+O_{L^\infty}(|K^\varepsilon|^\infty)$ for some constant $\alpha$ and some $k\in \llbracket 1,N \rrbracket$, then we have
\begin{align}\label{crochetbord2}
\left\langle \partial_n u_0^\varepsilon , \psi \right\rangle_{H^{-1/2}, H^{1/2}(\Gamma_{\mathcal D}^\varepsilon)}=\alpha \left\langle \partial_n u_0^\varepsilon , \1 \right\rangle_{\Gamma_k^\varepsilon}+O(|K^\varepsilon|^\infty).
\end{align}
\end{prop}
\begin{proof}
 Let us check that the distribution $\partial_n u_0^\varepsilon$ is
actually a positive distribution. Indeed, let $\phi \in
H^{1/2}(\partial \Omega)$ be such that $\phi \ge 0$. Let us now
consider $\Phi \in H^1(\Omega)$ the solution to
$$\left\{
  \begin{aligned}
    -\Delta \Phi &=0 \text{ in } \Omega \\
    \Phi &= \phi \text{ on } \Gamma^\varepsilon_{\mathcal D}\\
   \partial_n \Phi &= 0 \text{ on } \Gamma^\varepsilon_{\mathcal N}.
    \end{aligned}
\right.
$$
By the maximum principle, one has $\Phi \ge 0$ in $\Omega$. Let us
recall the argument. Let us consider $\Phi_- = \max(-\Phi,0) \in
H^1(\Omega)$. Notice that, using the trace theorem, $\Phi_-=0$ almost everywhere on
$\Gamma^\varepsilon_{\mathcal D}$. Thus, by integration by parts, one
has: $\int_\Omega |\nabla \Phi_-|^2= -\int_\Omega \nabla \Phi \cdot \nabla \Phi_- =  \int_\Omega (\Delta \Phi) \Phi_-=0$ and
thus $\Phi_-=0$.

Now, one has, using integration by parts,
the facts that $-\Delta \Phi = 0 \in L^2(\Omega)$, $\int_{\partial \Omega} u_0^\varepsilon \partial_n \Phi = 0$, $u_0^\varepsilon$ is an eigenfunction, 
$u_0^\varepsilon \le 0$, and $\lambda_0^\varepsilon > 0$,
\begin{align*}
\langle \partial_n u_0^\varepsilon, \phi \rangle_{\partial \Omega} &= \int_{\Omega} \nabla u_0^\varepsilon \cdot
                                                   \nabla \Phi + \int_{\Omega} \Delta u_0^\varepsilon \Phi\\
  &= \int_{\Omega} u_0^\varepsilon \Delta \Phi - \lambda^\varepsilon
    _0 \int_{\Omega} u_0^\varepsilon \Phi  \ge 0.
  \end{align*}
This shows that $\partial_n u_0^\varepsilon$ is a positive
distribution. Therefore, from the Riesz–Markov–Kakutani representation
theorem (see for example~\cite{rudin1974real} or~\cite{strichartz2003guide}), one can identify
$\partial_n u_0^\varepsilon/\lambda_0^\varepsilon$ as a probability measure (since $\langle
\partial_n u_0^\varepsilon, 1 \rangle = \lambda_0^\varepsilon$ since
we choose the normalization $\int_\Omega u_0^\varepsilon =
-1$). Moreover, the support of this measure is
$\overline{\Gamma^\varepsilon_{\mathcal D}}$ from the boundary conditions
  satisfied by $u_0^\varepsilon$ (see~\eqref{eq:u0}).
  
For the last statement, it suffices to write $\psi=\alpha \1_{\Gamma_k^\varepsilon}+\tilde \psi$ with $\|\tilde \psi\|_\infty=O(|K^\varepsilon|^\infty)$ and apply \eqref{eq:dnu_mes_proba}:
\begin{align}
\left\langle \partial_n u_0^\varepsilon , \psi
  \right\rangle_{H^{-1/2}, H^{1/2}(\Gamma_{\mathcal D}^\varepsilon)}&=\alpha
                                                           \left\langle
                                                           \partial_n
                                                           u_0^\varepsilon
                                                           ,
                                                           \1_{\Gamma_k^\varepsilon}
                                                           \right\rangle_{H^{-1/2},
                                                           H^{1/2}(\Gamma_{\mathcal D}^\varepsilon)}
                                                           +
                                                           \left\langle
                                                           \partial_n
                                                           u_0^\varepsilon
                                                           , \psi - \alpha
                                                           \1_{\Gamma_k^\varepsilon}
                                                           \right\rangle_{H^{-1/2},
                                                           H^{1/2}(\Gamma_{\mathcal D}^\varepsilon)}
                                                           \\
		& =\alpha \left\langle \partial_n u_0^\varepsilon , \1
           \right\rangle_{\Gamma_k^\varepsilon} +\lambda_0^\varepsilon
           \, O(|K^\varepsilon|^\infty).
\end{align}
Combined with Theorem \ref{gap}, this gives the result.
\end{proof}

This result can also be obtained by using a probabilistic
representation of the function $u_0^\varepsilon$ and its normal
derivative $\partial_n u_0^\varepsilon$, as we explain, now, referring
to~\cite{Louis} for detailed proofs. Let us define the Brownian motion in $\Omega$ reflected on $\partial \Omega$
\begin{align}\label{process}
dX_t=\sqrt 2 \, dB_t- \1_{\partial \Omega
}(X_t) n_{X_t} \, d L_t
\end{align}
where $L_t$ is the local time on $\partial \Omega$.
The first exit time is
$$\tau:= \inf \{t\geq 0\, ; \, X_t \in \Gamma_{\mathcal D}^\varepsilon
\}.$$
We recall Ito's generalized formula for functions of the process
\eqref{process} starting at $x\in \Omega$: 
for a $\mathcal C^{1,2}(\Omega)$ function $w$ and $t\in [0, \tau]$, it holds 
\begin{align}\label{ito}
w(t,X_t^x)-w(0,x&)=\int_0^t \big( \partial_t w+\Delta w\big) (s,X_s^x)
   \D s+\sqrt 2 \int_0^t \nabla w(s,X_s^x)   \D B_s\\
  &\quad - \int_0^t \1_{\Gamma_{\mathcal N}^\varepsilon}(X_s) \partial_n w(s,X_s^x)   \D L_s.
\end{align}
Finally, we denote the probability measure on $\Omega$:
$$\nu_0^\varepsilon(\D x)=-u_0^\varepsilon(x) \D x.$$

\begin{prop}
The process $(X_t)_{t \ge 0}$ absorbed on
$\Gamma^\varepsilon_{\mathcal D}$ has a unique quasi-stationary
distribution, which is $\nu_0^\varepsilon$: for all $t>0$,  for any test function $\Psi \in \mathcal
C^\infty(\Omega)$,
$$\E_{\nu_0^\varepsilon}(\Psi(X_t) 1_{t < \tau})=\left( \int \Psi \D
  \nu_0^\varepsilon\right) 
\mathbb P_{\nu_0^\varepsilon}(t < \tau),$$
where the subscript $\nu_0^\varepsilon$ indicates that $X_0
\sim \nu_0^\varepsilon$.
Moreover, if $X_0 \sim \nu_0^\varepsilon$, $\tau$ is
exponentially distributed with parameter $\lambda_0^\varepsilon$, $\tau$ is
independent of $X_\tau$ and
the law of $X_\tau$ is characterized by:
for any $\psi\in \mathcal C^\infty(\Gamma_{\mathcal D}^\varepsilon)$,
\begin{align}\label{crochetbord}
\E_{\nu_0^\varepsilon} \Big( \psi(X_\tau) \Big)
  =\frac{1}{\lambda_0^\varepsilon} \left\langle \partial_n
  u_0^\varepsilon , \psi \right\rangle_{\partial \Omega}.
\end{align}
\end{prop}
\begin{proof}
The existence and uniqueness of the QSD, as well as the fact that it
is $-u_0^\varepsilon(x) \D x$ is proven in~\cite[Appendix
B]{Louis}. The fact that $\tau$ is exponentially distributed with the parameter $\lambda_0^\varepsilon$ and
independent of $X_\tau$ then follows from standard arguments on QSDs
(see e.g.~\cite{collet2013quasi}). 
It remains to identify the first exit point distribution.

Let $\psi\in  \mathcal C^{\infty}(\Gamma_{\mathcal D}^\varepsilon)$.
One can show using \eqref{ito} and regularity results for mixed
boundary problems that the function $\Psi:x\mapsto \E_x (\psi(X_\tau))$ satisfies
\begin{equation}
\label{eq:psi}
\left \{
    \begin{aligned}
        \Delta \Psi & = 0
        &&\text{ in }  \Omega\\
        \partial_n \Psi & = 0 &&\text{ on } \Gamma^\varepsilon_{\mathcal N} \\
	\Psi &=\psi &&\text{ on }  \Gamma^\varepsilon_{\mathcal D}.
    \end{aligned}
    \right . 
  \end{equation}
In particular (see again~\cite[Appendix A]{Louis}), it can be shown that
$\Psi$  is $\mathcal C^\infty(\overline{\Omega} \setminus
\partial \Gamma_{\mathcal D})$ and $\mathcal C^0(\overline{\Omega})$.
The results then follow after integrating $\Psi$ against $\nu_0^\varepsilon$, writing
$u_0^\varepsilon=\frac{-1}{\lambda_0^\varepsilon} \Delta
u_0^\varepsilon $ and using Green's formula.
\end{proof}

\section{Proofs of Lemma \ref{f0} and Corollary \ref{corf0}}\label{appf0}

\hop
\textit{Proof of Lemma \ref{f0}}.
Let us start with the function $S_k^\varepsilon$.
We have
\begin{align}\label{decompoS}
\|S_k^\varepsilon-S_k^0\|^2\leq \int_{\Omega\backslash B(x^{(k)}, \varepsilon_k^{2/3})} |S_k^\varepsilon(x)-S_k^0(x) |^2 \D x + \int_{B(x^{(k)}, \varepsilon_k^{2/3})} |S_k^\varepsilon(x)-S_k^0(x) |^2 \D x.
\end{align}
In order to estimate the first term, notice that for all $x\in \Omega\backslash B(x^{(k)}, \varepsilon_k^{2/3})$, it holds
\begin{align}
S_k^\varepsilon(x)-S_k^0(x)&=\frac{1}{\|w_k^\varepsilon\|_1}\int_{T^-_{k,\varepsilon}}^{T^+_{k,\varepsilon}} w_k^\varepsilon(t) \ln \left(\frac{|x-\psi_k(t)|}{|x-x^{(k)}|}\right) \D t \\
		&=\frac{1}{\|w_k^\varepsilon\|_1}\int_{T^-_{k,\varepsilon}}^{T^+_{k,\varepsilon}} w_k^\varepsilon(t) \ln \left(\frac{|x-x^{(k)}+x^{(k)}-\psi_k(t)|}{|x-x^{(k)}|}\right) \D t\\
		&=\frac{1}{\|w_k^\varepsilon\|_1}\int_{T^-_{k,\varepsilon}}^{T^+_{k,\varepsilon}} w_k^\varepsilon(t) \ln \left(\frac{|x-x^{(k)}|+O(\varepsilon_k)}{|x-x^{(k)}|}\right) \D t\\
		&=\frac{1}{\|w_k^\varepsilon\|_1}\int_{T^-_{k,\varepsilon}}^{T^+_{k,\varepsilon}} w_k^\varepsilon(t) \ln \left(1+O(\varepsilon_k^{1/3})\right) \D t=O(\varepsilon_k^{1/3})  \label{Sloin}
\end{align}
and the first term of \eqref{decompoS} thus satisfies
\begin{align}\label{inteloin}
\int_{\Omega\backslash B(x^{(k)}, \varepsilon_k^{2/3})} |S_k^\varepsilon(x)-S_k^0(x) |^2 \D x =O(\varepsilon_k^{2/3}).
\end{align}
For the second term, we can write using Minkowski's integral inequality
\begin{align}
&\int_{B(x^{(k)}, \varepsilon_k^{2/3})} |S_k^\varepsilon(x)-S_k^0(x)
                |^2 \D x \leq C \int_{B(x^{(k)}, \varepsilon_k^{2/3})}
                |S_k^\varepsilon(x)|^2+|S_k^0(x) |^2 \D x \\
&		\leq \frac{C}{\|w_k^\varepsilon\|_1} \int_{T^-_{k,\varepsilon}}^{T^+_{k,\varepsilon}}w_k^\varepsilon(t) \int_{B(x^{(k)}, \varepsilon_k^{2/3})} \ln \Big(|x-\psi_k(t)|\Big)^2 \D x \, \D t  +C \int_{B(x^{(k)}, \varepsilon_k^{2/3})} \ln \Big(|x-x^{(k)}|\Big)^2 \D x\\
		&\leq \frac{C}{\|w_k^\varepsilon\|_1} \int_{T^-_{k,\varepsilon}}^{T^+_{k,\varepsilon}}w_k^\varepsilon(t) \int_{B(\psi_k(t), 2\varepsilon_k^{2/3})} \ln \Big(|x-\psi_k(t)|\Big)^2 \D x \, \D t  +C \int_{B(x^{(k)}, \varepsilon_k^{2/3})} \ln \Big(|x-x^{(k)}|\Big)^2 \D x
\end{align}
Using now polar coordinates centered at $\psi_k(t)$ (resp. at $x^{(k)}$), we obtain
\begin{align}
&\int_{B(x^{(k)}, \varepsilon_k^{2/3})} |S_k^\varepsilon(x)-S_k^0(x)
                |^2 \D x\\
  &\leq \frac{C}{\|w_k^\varepsilon\|_1} \int_{T^-_{k,\varepsilon}}^{T^+_{k,\varepsilon}}w_k^\varepsilon(t) \int_0^{2\pi} \int_0^{2\varepsilon_k^{2/3}} (\ln r)^2  r\,  \D r \, \D \theta \, \D t +C \int_0^{2\pi} \int_0^{\varepsilon_k^{2/3}} (\ln r)^2  r\,  \D r \, \D \theta\\
		&\leq C \int_0^{2\varepsilon_k^{2/3}} (\ln r)^2  r\,  \D r \leq C \varepsilon_k^{2/3} \label{intepres}.
\end{align}
Combining \eqref{decompoS}, \eqref{inteloin} and \eqref{intepres} shows the result for the function $S_k^\varepsilon$.\\
Now for the function $R_k^\varepsilon$, notice that thanks to
\eqref{defR}, it is sufficient to show that for all $\sigma\in \left[\frac {T^-_{k,\varepsilon}}{\varepsilon_k}, \frac{T^+_{k,\varepsilon}}{\varepsilon_k}\right]$, we have 
\begin{align}\label{amqR-R0}
\|\tilde R_k(\cdot,\varepsilon_k \sigma)-R_k^0   \|_{W^{1,3}}=O(\varepsilon_k^{1/3}).
\end{align}
Now thanks to \eqref{eq:Rkepd2} and \eqref{defR0}, for all $\sigma\in \left[\frac{T^-_{k,\varepsilon}}{\varepsilon_k}, \frac{T^+_{k,\varepsilon}}{\varepsilon_k}\right]$, the function $\tilde R_k(\cdot,\varepsilon_k \sigma)-R_k^0$ satisfies
\begin{equation}
\left \{
    \begin{aligned}
        \Delta_x \left( \tilde R_k(x,\varepsilon_k \sigma) -R_k^0(x) \right) & =0
        &&
x\in \Omega, \\
        \partial_{n_x} \left( \tilde R_k(x,\varepsilon_k \sigma) -R_k^0(x) \right) & = N_k(x,0) -N_k(x,\varepsilon_k \sigma) &&
  x\in \partial \Omega.
    \end{aligned}
    \right . 
\end{equation}
Thus, since $\tilde R_k(x^{(k)},\varepsilon_k \sigma)
-R_k^0(x^{(k)})=0$ (see Proposition~\ref{solR}), according to \cite[Proposition 2.3]{Aramaki}, it is sufficient to prove that for all $\sigma\in \left[\frac{T^-_{k,\varepsilon}}{\varepsilon_k}, \frac{T^+_{k,\varepsilon}}{\varepsilon_k}\right]$,
\begin{align}\label{amqN-N0}
N_k(\cdot,0) -N_k(\cdot,\varepsilon_k \sigma)=O(\varepsilon_k^{1/3}) \qquad \text{in }L^\infty(\partial \Omega)
\end{align}
to obtain \eqref{amqR-R0}.
Let us then establish \eqref{amqN-N0}.
For $x\in\partial \Omega \backslash B(x^{(k)}, \varepsilon_k^{1/3})$, it holds
\begin{align}
N_k(x,\varepsilon_k \sigma)&=\frac{n_x\cdot \left(x-x^{(k)}+x^{(k)}-\psi_k(\varepsilon_k \sigma)\right)}{\left|x-x^{(k)}+x^{(k)}-\psi_k(\varepsilon_k \sigma)\right|^2}=\frac{n_x\cdot \left(x-x^{(k)}\right)+O(\varepsilon_k)}{\left|x-x^{(k)}\right|^2+O\left( \left|x-x^{(k)}\right| \varepsilon_k  \right)}\\
		&=\frac{n_x\cdot \left(x-x^{(k)}\right)+O(\varepsilon_k)}{\left|x-x^{(k)}\right|^2} \left( 1 +O\left( \frac{\varepsilon_k}{|x-x^{(k)}|}  \right) \right)=\left(N_k(x,0)+O\left( \varepsilon_k^{1/3}  \right) \right) \left( 1 +O\left( \varepsilon_k^{2/3}  \right) \right)\\
		&=N_k(x,0)+O\left( \varepsilon_k^{1/3}  \right).
\end{align}
On the other hand, for $x\in \partial \Omega \cap B(x^{(k)}, \varepsilon_k^{1/3})$, putting $x= \psi_k(\varepsilon_k^{1/3}\theta)$ with $\theta \in \left[\frac{T^-_{k,\varepsilon}}{\varepsilon_k^{1/3}}, \frac{T^+_{k,\varepsilon}}{\varepsilon_k^{1/3}}\right]$, we have for all $\sigma\in \left[\frac{T^-_{k,\varepsilon}}{\varepsilon_k}, \frac{T^+_{k,\varepsilon}}{\varepsilon_k}\right]$,
$$x-\psi_k(\varepsilon_k\sigma)=\psi_k'(\varepsilon_k^{1/3}\theta)\left( \varepsilon_k^{1/3}\theta -\varepsilon_k\sigma \right)+\frac12 \psi_k''(\varepsilon_k^{1/3}\theta)\left( \varepsilon_k^{1/3}\theta -\varepsilon_k\sigma \right)^2+O\left( \left( \varepsilon_k^{1/3}\theta -\varepsilon_k\sigma \right)^3 \right).$$
Since moreover $n_x \perp \psi_k'(\varepsilon_k^{1/3}\theta)$, this gives
\begin{align}
N_k(x,\varepsilon_k \sigma)&=\frac{n_x \cdot \psi_k''(\varepsilon_k^{1/3}\theta)\varepsilon_k^{2/3} \left( \theta -\varepsilon_k^{2/3}\sigma \right)^2+O\left(  \varepsilon_k \left( \theta -\varepsilon_k^{2/3}\sigma \right)^3 \right)}{2|\psi_k'(\varepsilon_k^{1/3}\theta)|^2 \varepsilon_k^{2/3} \left( \theta -\varepsilon_k^{2/3}\sigma \right)^2+O\left(  \varepsilon_k \left( \theta -\varepsilon_k^{2/3}\sigma \right)^3 \right)}\\
		&=\frac{n_x \cdot \psi_k''(\varepsilon_k^{1/3}\theta)+O\left(  \varepsilon_k^{1/3} \left( \theta -\varepsilon_k^{2/3}\sigma \right) \right)}{2|\psi_k'(\varepsilon_k^{1/3}\theta)|^2 +O\left(  \varepsilon_k^{1/3} \left( \theta -\varepsilon_k^{2/3}\sigma \right) \right)}\\
		&=\frac{n_x \cdot \psi_k''(\varepsilon_k^{1/3}\theta)}{2|\psi_k'(\varepsilon_k^{1/3}\theta)|^2 } +O\left(  \varepsilon_k^{1/3} \right) \label{Ntheta}
\end{align}
and therefore
$$N_k(x,\varepsilon_k \sigma)-N_k(x,0)=O\left(  \varepsilon_k^{1/3} \right)$$
as the leading term of \eqref{Ntheta} does not depend on $\sigma$.
This shows \eqref{amqN-N0} and thus concludes the proof of Lemma \ref{f0}. \hfill $\Box$

\hop
\textit{Proof of Corollary \ref{corf0}}.
Let us start with the statement on $S_k^\varepsilon$.
When $k=j$, it is directly given by Lemma \ref{Sbord}.
When $k\neq j$, one can simply proceed as in \eqref{Sloin} : for all $x\in \Gamma_j^\varepsilon$, we have
\begin{align}
S_k^\varepsilon(x)-S_k^0(x^{(j)})&=\frac{1}{\|w_k^\varepsilon\|_1}\int_{T^-_{k,\varepsilon}}^{T^+_{k,\varepsilon}} w_k^\varepsilon(t) \ln \left(\frac{|x-\psi_k(t)|}{|x^{(k)}-x^{(j)}|}\right) \D t \\
		&=\frac{1}{\|w_k^\varepsilon\|_1}\int_{T^-_{k,\varepsilon}}^{T^+_{k,\varepsilon}} w_k^\varepsilon(t) \ln \left(\frac{|x-x^{(j)}+x^{(j)}-x^{(k)}+x^{(k)}-\psi_k(t)|}{|x^{(k)}-x^{(j)}|}\right) \D t\\
		&=\frac{1}{\|w_k^\varepsilon\|_1}\int_{T^-_{k,\varepsilon}}^{T^+_{k,\varepsilon}} w_k^\varepsilon(t) \ln \left(\frac{|x^{(j)}-x^{(k)}|+O(|\varepsilon|)}{|x^{(k)}-x^{(j)}|}\right) \D t\\
		&=\frac{1}{\|w_k^\varepsilon\|_1}\int_{T^-_{k,\varepsilon}}^{T^+_{k,\varepsilon}} w_k^\varepsilon(t) \ln \left(1+O(|\varepsilon|)\right) \D t=O(|\varepsilon|).
\end{align}
For the estimate on $R_k^\varepsilon$, we have by Lemma \ref{f0} and
thanks to the Sobolev injection $W^{1,3}(\Omega) \subset \mathcal C^{1/3}(\overline{\Omega})$ that
$$\|R_k^\varepsilon-R_k^0\|_{\mathcal C^{1/3}(\overline\Omega)}=O\left( \varepsilon_k^{1/3} \right).$$
Thus, it holds for all $x\in \Gamma_j^\varepsilon$
$$R_k^\varepsilon(x)-R_k^0(x^{(j)})=R_k^\varepsilon(x)- R_k^\varepsilon(x^{(j)}) +R_k^\varepsilon(x^{(j)}) -R_k^0(x^{(j)})=O\left( \varepsilon_j^{1/3}\right)+O\left( \varepsilon_k^{1/3}\right)$$
thanks also to Corollary \ref{Rhold}, and the result follows since $R_k^0(x^{(k)})=0$.
\hfill $\Box$

\section{Study of $D_M^\varepsilon$}

\begin{lem}\label{Dinv}
Recall the matrix $D_M^\varepsilon$ introduced in \eqref{DM} as well as the notation \eqref{dtilde} and let $k_0\in \llbracket 1,N\rrbracket$.
It holds 
$$\left(  \frac{\pi}{|\Omega|}\overline{K^\varepsilon} \right)^M- \frac{\pi}{|\Omega|}  \mathop{\sum_{\gamma \in \N^{N};\gamma_{k_0}=1}}_{\substack{1\leq |\gamma|\leq M}} (K^\varepsilon)^{\gamma} \left(\frac{\pi}{|\Omega|}\overline{K^\varepsilon}\right)^{M-|\gamma|}  
 \tilde d^{(k_0)}
(\gamma-e_{k_0}) =  \mathop{\sum_{\gamma \in \N^{N};\gamma_{k_0}=0}}_{\substack{ |\gamma|= M}}
 \tilde d^{(k_0)}(\gamma) (K^\varepsilon)^{\gamma}
  $$
In particular, all the diagonal entries of $D_M^\varepsilon$ are
positive (see~\eqref{dtildepos}).
\end{lem}

\begin{proof}
Let us proceed by induction on $M\geq 1$.

When $M=1$, the statement is simply thanks to \eqref{cdl0} 
$$\frac{\pi}{|\Omega|}\overline{K^\varepsilon}-\frac{\pi}{|\Omega|}K^\varepsilon_{k_0}=  \frac{\pi}{|\Omega|}\sum_{k\neq k_0} K_k^\varepsilon$$
and clearly holds true.

Suppose now that the statement holds true for some $M\geq 1$ and let us compute
\begin{align}
&\left(  \frac{\pi}{|\Omega|}\overline{K^\varepsilon} \right)^{M+1}- \frac{\pi}{|\Omega|}  \mathop{\sum_{\gamma \in \N^{N};\gamma_{k_0}=1}}_{\substack{1\leq |\gamma|\leq M+1}} (K^\varepsilon)^{\gamma} \left(\frac{\pi}{|\Omega|}\overline{K^\varepsilon}\right)^{M+1-|\gamma|}  
 \tilde d^{(k_0)}
(\gamma-e_{k_0})  \\
		& =  \frac{\pi}{|\Omega|}\overline{K^\varepsilon}   \left[  \left(  \frac{\pi}{|\Omega|}\overline{K^\varepsilon} \right)^{M}- \frac{\pi}{|\Omega|}  \mathop{\sum_{\gamma \in \N^{N};\gamma_{k_0}=1}}_{\substack{1\leq |\gamma|\leq M}} (K^\varepsilon)^{\gamma} \left(\frac{\pi}{|\Omega|}\overline{K^\varepsilon}\right)^{M-|\gamma|}  
 \tilde d^{(k_0)}
(\gamma-e_{k_0}) \right]\\
		& \quad  -  \frac{\pi}{|\Omega|}  \mathop{\sum_{\gamma \in \N^{N};\gamma_{k_0}=1}}_{\substack{ |\gamma|= M+1}} (K^\varepsilon)^{\gamma} 
 \tilde d^{(k_0)}
(\gamma-e_{k_0}) \\
		& = \frac{\pi}{|\Omega|}\overline{K^\varepsilon}     \mathop{\sum_{\gamma \in \N^{N};\gamma_{k_0}=0}}_{\substack{ |\gamma|= M}}
 \tilde d^{(k_0)}(\gamma) (K^\varepsilon)^{\gamma}     
-  \frac{\pi}{|\Omega|}  \mathop{\sum_{\gamma \in \N^{N};\gamma_{k_0}=1}}_{\substack{ |\gamma|= M+1}} (K^\varepsilon)^{\gamma} 
 \tilde d^{(k_0)}
(\gamma-e_{k_0})  \label{hr}\\
			&  =\frac{\pi}{|\Omega|}  \left( \overline{K^\varepsilon} -K^\varepsilon_{k_0}   \right)    \mathop{\sum_{\gamma \in \N^{N};\gamma_{k_0}=0}}_{\substack{ |\gamma|= M}}
 \tilde d^{(k_0)}(\gamma) (K^\varepsilon)^{\gamma}     
+  \frac{\pi}{|\Omega|} \mathop{\sum_{\gamma \in \N^{N};\gamma_{k_0}=0}}_{\substack{ |\gamma|= M}}
 \tilde d^{(k_0)}(\gamma) (K^\varepsilon)^{\gamma+e_{k_0}} \label{2dern}   \\
		& \quad
-  \frac{\pi}{|\Omega|}  \mathop{\sum_{\gamma \in \N^{N};\gamma_{k_0}=1}}_{\substack{ |\gamma|= M+1}} (K^\varepsilon)^{\gamma}  
 \tilde d^{(k_0)}
(\gamma-e_{k_0}) 
\end{align}
where we used the induction hypothesis to establish \eqref{hr}.
The 
change of index $\beta=\gamma+e_ {k_0}$ shows that the last two terms of \eqref{2dern} actually cancel one another.
It thus remains to prove that
\begin{align}\label{amqappB}
\frac{\pi}{|\Omega|}  \left( \overline{K^\varepsilon} -K^\varepsilon_{k_0}   \right)    \mathop{\sum_{\gamma \in \N^{N};\gamma_{k_0}=0}}_{\substack{ |\gamma|= M}}
 \tilde d^{(k_0)}(\gamma) (K^\varepsilon)^{\gamma} = \mathop{\sum_{\gamma \in \N^{N};\gamma_{k_0}=0}}_{\substack{ |\gamma|= M+1}}
 \tilde d^{(k_0)}(\gamma) (K^\varepsilon)^{\gamma}     .
\end{align}
By the change of index $\beta=\gamma+e_ {k}$, we have
\begin{align}
&\frac{\pi}{|\Omega|}  \left( \overline{K^\varepsilon} -K^\varepsilon_{k_0}   \right)    \mathop{\sum_{\gamma \in \N^{N};\gamma_{k_0}=0}}_{\substack{ |\gamma|= M}}
 \tilde d^{(k_0)}(\gamma) (K^\varepsilon)^{\gamma} = \frac{\pi}{|\Omega|}  \sum_{k\neq k_0}  \mathop{\sum_{\gamma \in \N^{N};\gamma_{k_0}=0}}_{\substack{ |\gamma|= M}}
 \tilde d^{(k_0)}(\gamma) (K^\varepsilon)^{\gamma+e_k} \\
		&=\frac{\pi}{|\Omega|}  \sum_{k\neq k_0}  \mathop{\sum_{\beta \in \N^{N};\beta_{k_0}=0; \beta_{k}\neq 0 }}_{\substack{ |\beta|= M+1}}
 \tilde d^{(k_0)}(\beta-e_k) (K^\varepsilon)^{\beta} =\frac{\pi}{|\Omega|}  \mathop{\sum_{\beta \in \N^{N};\beta_{k_0}=0 }}_{\substack{ |\beta|= M+1}}  \sum_{k\in \mathrm{supp}\, \beta}  
 \tilde d^{(k_0)}(\beta-e_k) (K^\varepsilon)^{\beta} \\
			&= \mathop{\sum_{\beta \in \N^{N};\beta_{k_0}=0 }}_{\substack{ |\beta|= M+1}} \Bigg(  \mathop{\sum_{\gamma \in \N^{N};\gamma\leq \beta }}_{\substack{ |\gamma|= |\beta|-1}}  \frac{\pi}{|\Omega|} 
 \tilde d^{(k_0)}(\gamma) \Bigg) (K^\varepsilon)^{\beta}.
\end{align}
The coefficient between the parenthesis appears to be exactly $\tilde d^{(k_0)}(\beta)$ thanks to \eqref{dn}.
This shows that \eqref{amqappB} holds true and therefore completes the
proof.
\end{proof}

\section{Multivariate polynomials}

\begin{lem}\label{multipoly}
Let $P\in \mathscr P^{hom}_n$ in $N$ variables.
Suppose that there exists $\mathcal O\subset \R^N$ a neighborhood of $0$ such that for all $X\in \mathcal O$ satisfying $X_1>0$, ..., $X_N>0$, it holds 
$$P(X)=O(|X|^{n+1}).$$
Then $P=0$.
\end{lem}

\begin{proof}
There exists some real coefficients $(p_\gamma)_{|\gamma|=n}$ such that
$$P(X)=\sum_{|\gamma|=n}p_\gamma X^\gamma.$$
Let $X\in \R^N$ such that $X_1>0$, ..., $X_N>0$ and $|X|=1$.
For $t>0$ small enough, we also have $tX\in \mathcal O$.
Thus,
$$t^n\sum_{|\gamma|=n}p_\gamma X^\gamma=O(t^{n+1}).$$
Dividing both sides by $t^n$ and taking the limit $t\to 0$ shows that
$$\sum_{|\gamma|=n}p_\gamma X^\gamma=0.$$
Therefore,
$$(p_\gamma)_{|\gamma|=n}\perp (X^\gamma)_{|\gamma|=n}\qquad \forall X_1>0, ..., X_N>0 \text{ with }|X|=1$$
and consequently $(p_\gamma)_{|\gamma|=n}=0$.
This completes the proof.
\end{proof}

\subsection*{Acknowledgements} 
We would like to thank Lo\"is Delande, Boris Nectoux and Gabriel
Stoltz for fruitful discussions.
This work was partially funded by the Agence Nationale de la Recherche, under grant ANR-19-CE40-0010-01 (QuAMProcs), and by the European Research Council (ERC) under the European Union's Horizon 2020 research and innovation programme (project EMC2, grant agreement No 810367).

\bibliography{NarrowEscape}

\end{document}